\documentclass[11pt]{amsart}

\usepackage{a4wide}
\usepackage{amssymb}
\usepackage{mathrsfs}
\usepackage{enumerate}
\usepackage{esint}
\usepackage{titletoc}
\usepackage[colorlinks=true, urlcolor=red, linkcolor=red, citecolor=blue]{hyperref}
\usepackage[nameinlink]{cleveref}

\theoremstyle{plain}
\newtheorem{theorem}{Theorem}[section]
\newtheorem{lemma}[theorem]{Lemma}
\newtheorem{corollary}[theorem]{Corollary}

\theoremstyle{definition}
\newtheorem{definition}[theorem]{Definition}
\newtheorem{remark}[theorem]{Remark}

\numberwithin{equation}{section}

\DeclareMathOperator*{\esssup}{ess\,sup}
\DeclareMathOperator*{\essinf}{ess\,inf}
\DeclareMathOperator*{\essliminf}{ess\,liminf}

\makeatletter
\@namedef{subjclassname@2020}{\textup{2020} Mathematics Subject Classification}
\makeatother

\title[Wolff potential estimates and Wiener criterion for nonlocal equations]{Wolff potential estimates and Wiener criterion for nonlocal equations with Orlicz growth}

\author{Minhyun Kim}
\address{Department of Mathematics \& Research Institute for Natural Sciences, Hanyang University, 04763 Seoul, Republic of Korea}
\email{minhyun@hanyang.ac.kr}

\author{Ki-Ahm Lee}
\address{Department of Mathematical Sciences \& Research Institute of Mathematics, Seoul National University, 08826 Seoul, Republic of Korea}
\email{kiahm@snu.ac.kr}

\author{Se-Chan Lee}
\address{Research Institute of Mathematics, Seoul National University, Seoul 08826, Republic of Korea}
\email{dltpcks1@snu.ac.kr}

\thanks{The research of K.-A. Lee is supported by the National Research Foundation of Korea (NRF) grant funded by the Korea government (MSIP): NRF-2020R1A2C1A01006256. M. Kim is supported by the research fund of Hanyang University (HY-202300000001143) and the National Research Foundation of Korea (NRF) grant funded by the Korea government (MSIT) (RS-2023-00252297). S.-C. Lee is supported by Basic Science Research Program through the National Research Foundation of Korea (NRF) funded by the Ministry of Education (2022R1A6A3A01086546).}

\begin{document}

\begin{abstract}
We prove the Wolff potential estimates for nonlocal equations with Orlicz growth. As an application, we obtain the Wiener criterion in this framework, which provides a necessary and sufficient condition for boundary points to be regular. Our approach relies on the fine analysis of superharmonic functions in view of nonlocal nonlinear potential theory.
\end{abstract}

\maketitle

\section{Introduction}

Nonlinear potential theory studies harmonic functions associated with nonlinear, possibly degenerate or singular equations. Its origin can be traced to the foundational research of Havin--Maz'ya~\cite{MH72}, where a systematic investigation of various potentials is carried out in order to analyze fine properties of harmonic functions. They introduced what is nowadays called the Wolff potential, which plays a fundamental role in the potential theory. One of the most powerful applications of Wolff potentials is the Wiener criterion for the $p$-Laplace equation. In \cite{Maz70}, Maz'ya obtained the sufficiency part of the Wiener criterion by using some quantities related to Wolff potentials. After a long while, Kilpel\"ainen--Mal\'y~\cite{KM94} proved pointwise estimates for solutions of $p$-Laplace-type equation in terms of the Wolff potential, finally establishing the necessity part of the Wiener criterion. See also \cite{KK10,TW02} for different approaches. We refer the reader to books by Heinonen--Kilpel\"ainen--Mal\'y~\cite{HKM06}, Mal\'y--Ziemer~\cite{MZ97}, and Adams--Hedberg~\cite{AH96} for a comprehensive description of nonlinear potential theory.

As a nonlocal counterpart of nonlinear potential theory, nonlocal nonlinear potential theory has recently begun to be actively studied. It was initiated in Kuusi--Mingione--Sire~\cite{KMS15}, where the pointwise bounds by Wolff potentials for SOLA (solutions obtained as limits of approximations) of fractional $p$-Laplace-type equations are obtained for $p \in (2-s/n, n/s)$, where $s \in (0,1)$ is the order of differentiability. The restriction $p>2-s/n$ is natural when we study SOLA. The range of $p \in (1, n/s]$ was recently covered in Kim--Lee--Lee~\cite{KLL23} for the pointwise upper estimate by using superharmonic functions instead of SOLA as in \cite{KM94}. This upper bound, together with several fine properties of superharmonic functions developed in Korvenp\"a\"a--Kussi--Palatucci~\cite{KKP17}, led us to the Wiener criterion for the fractional $p$-Laplace-type equations in \cite{KLL23}. However, the pointwise lower bound for the full range of $p>1$ was not known and it is provided here in a more general framework; see \Cref{thm-Wolff-lower}.

The aforementioned papers focus on the standard nonlinearity described by a polynomial of degree $p>1$. In this paper, we aim to extend the theory to a general class of nonlocal nonlinear operators with Orlicz growth. More precisely, building upon the results of Kim--Lee~\cite{KL23}, we prove pointwise upper and lower bounds of superharmonic functions associated with the operator \eqref{eq-op} below. We then establish the Wiener criterion, which characterizes a regular boundary point in terms of nonlocal nonlinear potential theory.

Let $G: [0, \infty)\to[0,\infty)$ be given by
\begin{equation*}
G(t) = \int_0^t g(\tau) \,\mathrm{d}\tau
\end{equation*}
for some non-decreasing, left-continuous function $g:[0, \infty) \to [0,\infty)$ which is not identically zero. Such a function is called a \emph{Young function}. We assume throughout the paper that $G$ is differentiable, $g=G'$ is strictly increasing, and there exist $1<p\leq q$ such that
\begin{equation}\label{eq-pq}
pG(t) \leq tg(t) \leq qG(t) \quad\text{for all}~ t \geq 0
\end{equation}
and
\begin{equation}\label{eq-q}
t \mapsto g(t)/t^{q-1}\quad\text{is non-increasing}.
\end{equation}
We also assume without loss of generality that $G(1)=1$.

We consider an operator
\begin{equation}\label{eq-op}
\mathcal{L}u(x) = 2\,\mathrm{p.v.} \int_{\mathbb{R}^n} g(|D^su|) \frac{D^su}{|D^su|} \frac{k(x, y)}{|x-y|^{n+s}} \,\mathrm{d}y,
\end{equation}
where
\begin{equation*}
D^su=D^su(x, y) = \frac{u(x)-u(y)}{|x-y|^s}
\end{equation*}
denotes a \emph{fractional gradient} of order $s \in (0,1)$ and $k: \mathbb{R}^n \times \mathbb{R}^n \to [0, \infty]$ is a measurable kernel satisfying the ellipticity condition $\Lambda^{-1} \leq k(x, y) \leq \Lambda$ a.e.\ in $\mathbb{R}^n \times \mathbb{R}^n$ for some $\Lambda \geq 1$.

Let us provide the definition of the Wolff potential which takes the Orlicz growth into account. It will play a crucial role in our analysis.

\begin{definition}[Wolff potential]
Let $\mu$ be a Borel measure with finite total mass in $\mathbb{R}^{n}$. The \emph{Wolff potential ${\bf W}^{\mu}_{s, G}$ of the measure $\mu$} is defined by 
\begin{equation}\label{eq-Wolff}
{\bf W}^{\mu}_{s, G}(x_0, R) = \int_0^R \rho^s g^{-1} \left( \frac{|\mu|(B_\rho(x_0))}{\rho^{n-s}} \right) \frac{\mathrm{d}\rho}{\rho}
\end{equation}
for $x_0 \in \mathbb{R}^n$ and $0<R\leq \infty$.
\end{definition}

The Wolff potential ${\bf W}^\mu_{s, G}$ was first introduced in Fiorenza--Prignet~\cite{FP03} as a generalization of the classical Wolff potential
\begin{equation*}
{\bf W}^{\mu}_{s, p}(x_0, R) = \int_0^R \left( \frac{|\mu|(B_\rho(x_0))}{\rho^{n-sp}} \right)^{1/(p-1)} \frac{\mathrm{d}\rho}{\rho}.
\end{equation*}

Our first main results are pointwise upper and lower bounds of superharmonic functions involving the Wolff potential \eqref{eq-Wolff}. See \Cref{def-superharmonic} for the definition of a superharmonic function.

\begin{theorem}[Pointwise upper bound]\label{thm-Wolff-upper}
Let $s \in (0,1)$. Let $u$ be a superharmonic function in $B_R(x_0)$ such that $u \geq 0$ in $B_R(x_0)$. Then $\mu := \mathcal{L}u$ is a nonnegative Borel measure and
\begin{equation}\label{eq-Wolff-upper}
u(x_0) \leq C \left( \inf_{B_{R/2}(x_0)}u + {\bf W}_{s, G}^{\mu}(x_0, R) + \mathrm{Tail}_g(u; x_0, R/2) \right)
\end{equation}
for some $C = C(n, p, q, s, \Lambda) > 0$.
\end{theorem}

Note that \Cref{thm-Wolff-upper} includes the existence of the measure $\mu=\mathcal{L}u$, which is called the \emph{Riesz measure}.

\begin{theorem}[Pointwise lower bound]\label{thm-Wolff-lower}
Let $s \in (0,1)$. Let $u$ be a superharmonic function in $B_R(x_0)$ such that $u \geq 0$ in $B_R$ and let $\mu = \mathcal{L}u$ be the Riesz measure. Then
\begin{equation}\label{eq-Wolff-lower}
 \mathbf{W}^{\mu}_{s, G}(x_0, R/2) \leq C(u(x_0) + \mathrm{Tail}_g(u_-; x_0, R))
\end{equation}
for some $C = C(n, p, q, s, \Lambda) > 0$.
\end{theorem}

In the standard growth case $G(t)=t^p$, on the one hand, the pointwise upper bound \eqref{eq-Wolff-upper} was proved by Kuusi--Mingione--Sire~\cite{KMS15} for SOLA when $p>2-s/n$ and Kim--Lee--Lee~\cite{KLL23} for superharmonic functions when $p \in (1, n/s]$ under the assumption that $\mu=\mathcal{L}u$ exists. On the other hand, the pointwise lower bound \eqref{eq-Wolff-lower} was only known for SOLA when $p \in (2-s/n, n/s)$. Hence, \Cref{thm-Wolff-upper} and \Cref{thm-Wolff-lower} are new even for the standard growth case because the existence of $\mu$ is proved and the full range of $p>1$ is covered. We refer the reader to Mal\'y~\cite{Mal03} and Chlebicka--Giannetti--Zatorska-Goldstein~\cite{CGZG20} for the Wolff potential estimates for local operators with Orlicz growth.



% ...

Let us next move on to the Wiener criterion, which provides a necessary and sufficient condition for a boundary point to be regular. It was first provided for the Laplacian by Wiener in his pioneering work \cite{Wie24}, and then extended in Littman--Stampacchia--Weinberger~\cite{LSW63} to general second-order linear operators in divergence form with measurable coefficients. The generalization to quasilinear operators of second-order has been done by Maz'ya~\cite{Maz70}, Gariepy--Ziemer~\cite{GZ77}, Lindqvist--Martio~\cite{LM85}, and Kilpel\"ainen--Mal\'y~\cite{KM94}. Further generalizations can be found in Labutin~\cite{Lab02} for $k$-Hessian operators, Alkhutov--Krasheninnikova~\cite{AK04} for $p(x)$-Laplacian, and Lee--Lee~\cite{LL21} for operators with Orlicz growth.

For nonlocal operators, the Wiener criterion has been established by Bliedtner--Hansen~\cite{BH86}, Hoh--Jacob~\cite{HJ96}, Landkof~\cite{Lan72}, and Bj\"orn~\cite{Bjo24} for the fractional Laplacian, and by Kim--Lee--Lee~\cite{KLL23} for nonlocal nonlinear operators with the standard $p$-growth. As an application of the Wolff potential estimates, we obtain the Wiener criterion for nonlocal Dirichlet problems associated with $\mathcal{L}$. To describe this, let $\Omega \subset \mathbb{R}^n$ be open and bounded. We say that a boundary point $x_0 \in \partial \Omega$ is \emph{regular} (with respect to $\mathcal{L}$) if, for each function $\vartheta \in V^{s, G}(\Omega) \cap C(\mathbb{R}^n)$, the unique harmonic function $u$ (with respect to $\mathcal{L}$) in $\Omega$ with $u-\vartheta \in V^{s, G}_0(\Omega)$ satisfies
\begin{equation}\label{eq-regular}
\lim_{\Omega \ni x \to x_0} u(x) = \vartheta(x_0).
\end{equation}
See \Cref{sec-spaces} for the definition of function spaces $V^{s, G}(\Omega)$ and $ V^{s, G}_0(\Omega)$, and \Cref{def-harmonic} for harmonic function. Note that the Dirichlet data $\vartheta$ is imposed in the complement of $\Omega$ because we are concerned with the nonlocal operator.

Our second main result is the following. The definition of $(s, G)$-capacity is given in \Cref{sec-sG-cap}.

\begin{theorem}[Wiener criterion] \label{thm-Wiener}
Let $s \in (0,1)$. Then, a boundary point $x_0 \in \partial \Omega$ is regular if and only if 
	\begin{equation}\label{eq-Wiener}
		\int_0^1 \rho^s g^{-1}\left(\frac{\mathrm{cap}_{s, G}(\overline{B_{\rho}(x_0)} \setminus \Omega, B_{2\rho}(x_0))}{\rho^{n-s}}\right) \frac{\mathrm{d}\rho}{\rho}=\infty.
	\end{equation}
\end{theorem}

\Cref{thm-Wiener} characterizes a regular boundary point in terms of the $(s, G)$-capacity, which is defined in \Cref{sec-capacity}. It says that a boundary point is regular if and only if the complement of $\Omega$ near the boundary point has enough $(s, G)$-capacity. Any domain satisfying the exterior ball condition, exterior cone condition, exterior $(\delta, R)$-Reifenberg flat condition, or exterior corkscrew condition satisfies \eqref{eq-Wiener}, since in all cases we can find a ball in $\overline{B_{\rho}(x_0)} \setminus \Omega$ with a radius comparable to $\rho$. Moreover, \Cref{thm-Wiener} implies that the regularity of a boundary point is completely determined by the local geometry of $\partial \Omega$. This phenomenon for nonlocal equations shows a clear contrast with what is observed in nonlocal minimal surface equations. In the case of nonlocal minimal surface equations, the boundary regularity is deeply influenced by the nonlocality of the equations; see Dipierro--Savin--Valdinoci~\cite{DSV17,DSV20a,DSV20b} for instance.

\begin{remark}\label{rmk-wiener}
Let us make further remarks on \Cref{thm-Wiener}.
\begin{enumerate}[(i)]
\item
\Cref{thm-Wiener} shows that the regularity of a boundary point depends only on $n$, $s$, and $g$, not on the ellipticity $\Lambda$.
\item
In the definition of regular boundary point, a Dirichlet data $\vartheta$ is not necessarily bounded nor compactly supported. In fact, the proof of \Cref{thm-Wiener} reveals that a boundary point $x_0$ is regular if and only if for each bounded $\vartheta \in V^{s, G}(\Omega) \cap C(\mathbb{R}^n)$ the unique harmonic function $u$ with $u-\vartheta \in V^{s, G}_0(\Omega)$ satisfies \eqref{eq-regular}. See \Cref{sec-spaces} for the definition of the function spaces.
\item
In our previous work \cite{KLL23}, we proved \Cref{thm-Wiener} when $G(t)=t^p$ by dividing the cases into two cases $p\leq n/s$ and $p>n/s$. However, such distinction is not available in the current framework since we treat operators with Orlicz growth. Instead, we consider two cases based on whether a single point has $(s, G)$-capacity zero or not. For the latter case, we introduce the notion of $(s, G)$-quasicontinuity to describe the regular behavior of $\mathcal{L}$-potential; see \Cref{sec-wiener} for details. Roughly speaking, this alternative viewpoint presents an implicit way to understand the `effective dimension', which is strongly related to $(s, G)$-capacity. Along the way, we also study several properties of potential theoretical tools such as $(s, G)$-capacity, $\mathcal{L}$-potential, and $(s, G)$-quasicontinuity in the Orlicz framework in \Cref{sec-capacity}, which are of independent interest.

We also point out that the case $p>n/s$ in \cite{KLL23} was treated incorrectly. The Sobolev embedding theorem was used to conclude that the harmonic function with $u-\vartheta \in V^{s, p}_0(\Omega)$ is continuous up to the boundary in this case, but the Sobolev embedding theorem requires some regularity of the boundary of $\Omega$. We correct this mistake here in this paper. Note that there is the same mistake in \cite{LL21}.
\end{enumerate}
\end{remark}

The paper is organized as follows. In \Cref{sec-preliminaries}, we provide several inequalities, introduce some function spaces, define supersolutions and superharmonic functions, and collect local estimates of supersolutions. \Cref{sec-upper} and \Cref{sec-lower} are devoted to the proofs of \Cref{thm-Wolff-upper} and \Cref{thm-Wolff-lower}, respectively. The methods are based on Kilpel\"ainen--Mal\'y~\cite{KM92,KM94} and Lukkari--Maeda--Marola~\cite{LMM10}. In \Cref{sec-capacity}, we define the $(s, G)$-capacity and $\mathcal{L}$-potential, which play a crucial role in potential theory, and provide several properties of them. Finally, in \Cref{sec-wiener}, we establish the Wiener criterion. Some algebraic inequalities can be found in \Cref{sec-ineq}.

\section{Preliminaries}\label{sec-preliminaries}

Recall that we assume that $G$ and $g$ satisfy \eqref{eq-pq} and \eqref{eq-q} throughout the paper. However, in this section, we are going to specify the conditions required in the lemmas and theorems. In fact, \Cref{thm-relation} (ii) and \Cref{thm-pointwise} are the only results in this preliminary section that require the assumption \eqref{eq-q}.

\subsection{Growth function}

In this section, we provide several algebraic inequalities involving growth functions $G$ and $g$, under the assumption \eqref{eq-pq}. One of the simplest inequalities is the following: for any $a, b \geq 0$ and $\varepsilon > 0$,
\begin{equation}\label{eq-alg}
	g(a)b \leq \varepsilon g(a)a + g(b/\varepsilon)b.
\end{equation}
The inequality \eqref{eq-alg} follows by splitting the cases $b \leq \varepsilon a$ and $b > \varepsilon a$.

It turned out in Kim--Lee~\cite{KL23} that it is convenient to introduce a function $\bar{g}(t)=G(t)/t$ because it satisfies
\begin{equation}\label{eq-bar-g-comp}
p\bar{g}(t) \leq g(t) \leq q\bar{g}(t), \quad p^{1/(q-1)} g^{-1}(t) \leq \bar{g}^{-1}(t) \leq q^{1/(p-1)} g^{-1}(t),
\end{equation}
and
\begin{equation}\label{eq-bar-g-pq}
(p-1)\bar{g}(t) \leq t\bar{g}'(t) \leq (q-1)\bar{g}(t).
\end{equation}
It is worth noticing that the condition \eqref{eq-bar-g-pq} with $\bar{g}$ replaced by $g$ does not follow from the assumption \eqref{eq-pq}. We import the following lemma from \cite{KL23}.

\begin{lemma}\label{lem-G}
Assume that $G$ and $g$ satisfy \eqref{eq-pq}. Let $t, t' \geq 0$ and $c=p^{1/(q-1)}/q^{1/(p-1)}$.
	\begin{enumerate}[(i)]
		\item
		For all $\lambda \geq 1$, it holds that
		\begin{alignat*}{2}
			&\lambda^p G(t) \leq G(\lambda t) \leq \lambda^q G(t),
			&\quad&\lambda^{1/q} G^{-1}(t) \leq G^{-1}(\lambda t) \leq \lambda^{1/p}G^{-1}(t), \\
			&\lambda^{p-1} \bar{g}(t) \leq \bar{g}(\lambda t) \leq \lambda^{q-1} \bar{g}(t),
			&&\lambda^{1/(q-1)} \bar{g}^{-1}(t) \leq \bar{g}^{-1}(\lambda t) \leq \lambda^{1/(p-1)} \bar{g}^{-1}(t), \\
			&\frac{p}{q} \lambda^{p-1} g(t) \leq g(\lambda t) \leq \frac{q}{p} \lambda^{q-1} g(t),
			&&c \lambda^{1/(q-1)} g^{-1}(t) \leq g^{-1}(\lambda t) \leq c^{-1} \lambda^{1/(p-1)} g^{-1}(t).
		\end{alignat*}
		\item
		For all $\lambda \leq 1$, it holds that
		\begin{alignat*}{2}
			&\lambda^q G(t) \leq G(\lambda t) \leq \lambda^p G(t),
			&\quad&\lambda^{1/p} G^{-1}(t) \leq G^{-1}(\lambda t) \leq \lambda^{1/q}G^{-1}(t), \\
			&\lambda^{q-1} \bar{g}(t) \leq \bar{g}(\lambda t) \leq \lambda^{p-1} \bar{g}(t),
			&&\lambda^{1/(p-1)} \bar{g}^{-1}(t) \leq \bar{g}^{-1}(\lambda t) \leq \lambda^{1/(q-1)} \bar{g}^{-1}(t), \\
			&\frac{p}{q} \lambda^{q-1} g(t) \leq g(\lambda t) \leq \frac{q}{p} \lambda^{p-1} g(t),
			&&c \lambda^{1/(p-1)} g^{-1}(t) \leq g^{-1}(\lambda t) \leq c^{-1} \lambda^{1/(q-1)} g^{-1}(t).
		\end{alignat*}
		\item
		It holds that
		\begin{alignat*}{2}
			&G(t+t') \leq 2^q(G(t)+G(t')),
			&\quad&G^{-1}(t+t') \leq 2^{1/p}(G^{-1}(t)+G^{-1}(t')), \\
			&\bar{g}(t+t') \leq 2^{q-1}(\bar{g}(t)+\bar{g}(t')),
			&&\bar{g}^{-1}(t+t') \leq 2^{1/(p-1)} (\bar{g}^{-1}(t) + \bar{g}^{-1}(t')), \\
			&g(t+t') \leq \frac{q}{p}2^{q-1}(g(t)+g(t')),
			&&g^{-1}(t+t') \leq c^{-1} 2^{1/(p-1)} (g^{-1}(t)+g^{-1}(t')).
		\end{alignat*}
	\end{enumerate}
\end{lemma}

We close this section with the following lemma, which is a generalization of Lemma~3 in Dyda~\cite{Dyd06}.

\begin{lemma}\label{lem-g-dyda}
Assume that $G$ and $g$ satisfy \eqref{eq-pq}. The following hold:
	\begin{enumerate}[(i)]
		\item 
		For any $m \in \mathbb{N}$ and $a_i \geq 0$ with $i=1, \dots, m$, we have
		\begin{equation*}
			\bar{g}^{-1}\left(\sum_{i=1}^m a_i\right) \leq m^{1/(p-1)} \sum_{i=1}^m \bar{g}^{-1}(a_i).
		\end{equation*}
	
		\item Let $\beta>1$. There exists a constant $C=C(p, q, \beta)>0$ such that for any $a_i \geq 0$ with $i=1, \dots$, we have
		\begin{equation*}
			g^{-1}\left(\sum_{i=1}^{\infty} a_i\right) \leq C\sum_{i=1}^{\infty}\beta^i g^{-1}(a_i).
		\end{equation*}
	\end{enumerate}
\end{lemma}

\begin{proof}
(i) We may assume without loss of generality that $a_1 \geq a_i$ for all $i=2, \dots, m$. By using \Cref{lem-G}~(i), we obtain
		\begin{equation*}
				\bar{g}^{-1}\left(\sum_{i=1}^m a_i\right) \leq \bar{g}^{-1}(ma_1) \leq m^{1/(p-1)}\bar{g}^{-1}(a_1) \leq m^{1/(p-1)} \sum_{i=1}^m \bar{g}^{-1}(a_i).
		\end{equation*}
(ii) By iterating the part (i) with $m=2$, we have
		\begin{equation}\label{eq-dyda}
				\bar{g}^{-1}\left(\sum_{k=0}^{\infty} c_k\right) \leq \sum_{k=0}^{\infty} \left( 2^{1/(p-1)} \right)^{k+1} \bar{g}^{-1}(c_k)
		\end{equation}
		wherever $c_k \geq 0$ with $k=0, 1, \dots$\,. Let $m \in \mathbb{N}$ be sufficiently large so that $2^{\frac{1}{m(p-1)}} \leq \beta$. Then \eqref{eq-dyda} and the part (i) imply
		\begin{align*}
			\bar{g}^{-1}\left(\sum_{i=1}^{\infty} a_i\right)=\bar{g}^{-1}\left(\sum_{k=0}^{\infty} \sum_{j=1}^ma_{km+j}\right) &\leq \sum_{k=0}^{\infty} \left( 2^{1/(p-1)} \right)^{k+1} \bar{g}^{-1} \left(\sum_{j=1}^m a_{km+j}\right)\\
			&\leq (2m)^{1/(p-1)} \sum_{k=0}^{\infty} 2^{k/(p-1)} \sum_{j=1}^m \bar{g}^{-1} \left(a_{km+j}\right)\\
			&\leq (2m)^{1/(p-1)} \sum_{i=1}^{\infty} \beta^{i} \bar{g}^{-1} (a_i).
		\end{align*}
	The desired inequality follows from \eqref{eq-bar-g-comp}.
\end{proof}

\subsection{Function spaces}\label{sec-spaces}

Assume that $G$ and $g$ satisfy \eqref{eq-pq}. Let $\Omega \subset \mathbb{R}^n$ be open. We define
\begin{align*}
	\varrho_{L^G(\Omega)}(u) &= \int_\Omega G(|u|) \,\mathrm{d}x, \\
	\varrho_{W^{s, G}(\Omega)}(u) &= \int_\Omega \int_\Omega G(|D^s u|) \frac{\mathrm{d}y \,\mathrm{d}x}{|x-y|^n}, \\
	\varrho_{V^{s, G}(\Omega)}(u) &= \int_\Omega \int_{\mathbb{R}^n} G(|D^s u|) \frac{\mathrm{d}y \,\mathrm{d}x}{|x-y|^n}.
\end{align*}
Then the \emph{Orlicz} and \emph{(fractional) Orlicz--Sobolev spaces} defined by
\begin{align*}
	L^G(\Omega) &= \left\{ u: \Omega \to \mathbb{R} ~\text{measurable}: \varrho_{L^G(\Omega)}(u) < \infty \right\}, \\
	W^{s, G}(\Omega) &= \left\{ u \in L^G(\Omega): \varrho_{W^{s, G}(\Omega)}(u) < \infty \right\}, \\
	V^{s, G}(\Omega) &= \left\{ u: \mathbb{R}^n \to \mathbb{R} ~\text{measurable}: u|_\Omega \in L^G(\Omega), \varrho_{V^{s, G}(\Omega)}(u) < \infty \right\},
\end{align*}
are Banach spaces with the norms
\begin{align*}
	\|u\|_{L^G(\Omega)} &:= \inf \left\{ \lambda>0: \varrho_{L^G(\Omega)}(u/\lambda) \leq 1 \right\}, \\
	\|u\|_{W^{s, G}(\Omega)} &:= \|u\|_{L^G(\Omega)} + [u]_{W^{s, G}(\Omega)} := \|u\|_{L^G(\Omega)} + \inf \left\{ \lambda>0: \varrho_{W^{s, G}(\Omega)}(u/\lambda) \leq 1 \right\}, \\
	\|u\|_{V^{s, G}(\Omega)} &:= \|u\|_{L^G(\Omega)} + [u]_{V^{s, G}(\Omega)} := \|u\|_{L^G(\Omega)} + \inf \left\{ \lambda>0: \varrho_{V^{s, G}(\Omega)}(u/\lambda) \leq 1 \right\},
\end{align*}
respectively. By $W^{s, G}_{\text{loc}}(\Omega)$ we denote the space of functions that belong to $W^{s, G}(D)$ for each open set $D \Subset \Omega$. We also define the space
\begin{equation*}
	 V^{s, G}_0(\Omega) = \overline{C_c^{\infty}(\Omega)}^{V^{s, G}(\Omega)},
\end{equation*}
which was denoted by $W^{s, p}_0(\Omega)$ in Kim--Lee--Lee~\cite{KLL23}. By \Cref{thm-Poincare} below, $[\cdot]_{V^{s, G}(\Omega)}$ is a norm on $V^{s, G}_0(\Omega)$. We refer the reader to \cite[Remark~2.2]{KLL23} for some remarks on these spaces.

\begin{theorem}[Fractional Poincar\'e inequality]\cite[Proposition~3.2]{Sal20}\label{thm-Poincare}
Assume that $G$ and $g$ satisfy \eqref{eq-pq}. Let $\Omega \subset \mathbb{R}^n$ be open and bounded. There exists a constant $C = C(n, p, q, s) > 0$ such that
\begin{equation*}
\int_\Omega G(|u|) \,\mathrm{d}x \leq C \int_{\Omega} \int_{\mathbb{R}^n} G(\mathrm{diam}^s(\Omega) |D^su|) \frac{\mathrm{d}y \,\mathrm{d}x}{|x-y|^n}
\end{equation*}
for all $u \in V_0^{s, G}(\Omega)$.
\end{theorem}

In order to study nonlocal equations, it is now standard to introduce the \emph{tail space}, which is defined by
\begin{equation*}
	L^g_s(\mathbb{R}^n) = \left\{ u \in L^g_{\text{loc}}(\mathbb{R}^n): \int_{\mathbb{R}^n} g\left(\frac{|u(x)|}{(1+|x|)^s}\right) \frac{\mathrm{d}x}{ (1+|x|)^{n+s} } <\infty \right\}
\end{equation*}
in the Orlicz setting. Equivalently, we can define
\begin{equation*}
	L^g_s(\mathbb{R}^n) = \left\{ u \in L^g_{\text{loc}} (\mathbb{R}^n): \mathrm{Tail}_g(u; x_0, R) <\infty ~\text{for some $x_0 \in \mathbb{R}^n$ and $R>0$} \right\},
\end{equation*}
where
\begin{equation*}
	\mathrm{Tail}_g(u; x_0, R) = R^s g^{-1} \left( R^s \int_{\mathbb{R}^n \setminus B_R(x_0)} g \left( \frac{|u(x)|}{|x-x_0|^s} \right) \frac{\mathrm{d}x}{|x-x_0|^{n+s}} \right)
\end{equation*}
is a \emph{(nonlocal) tail}. Note that $\mathrm{Tail}_g(u; x_0, R)$ is finite for any $x_0 \in \mathbb{R}^n$ and $R>0$ if $u \in L^g_s(\mathbb{R}^n)$. As we often analyze the integral in the tail, we denote
\begin{equation*}
	T(u; x_0, R) := g\left( \frac{\mathrm{Tail}_g(u; x_0, R)}{R^s} \right) = R^s \int_{\mathbb{R}^n \setminus B_R(x_0)} g \left( \frac{|u(x)|}{|x-x_0|^s} \right) \frac{\mathrm{d}x}{|x-x_0|^{n+s}}
\end{equation*}
for notational convenience.

\subsection{Supersolutions and superharmonic functions}\label{sec-super}

We recall the definitions and some properties of supersolutions and superharmonic functions defined in Kim--Lee~\cite{KL23}. Assume that $G$ and $g$ satisfy \eqref{eq-pq}. Let us begin with supersolutions. For measurable functions $u, v: \mathbb{R}^n \to \mathbb{R}$, we consider a quantity
\begin{equation*}
	\mathcal{E}(u, v) = \int_{\mathbb{R}^n} \int_{\mathbb{R}^n} g(|D^su|) \frac{D^su}{|D^su|} D^sv \frac{k(x, y)}{|x-y|^n} \,\mathrm{d}y \,\mathrm{d}x.
\end{equation*}

\begin{definition}\label{def-supersolution}
	A function $u \in W_{\text{loc}}^{s, G}(\Omega)$ with $u_- \in L^g_s(\mathbb{R}^n)$ is a \emph{(weak) supersolution} of $\mathcal{L}u=0$ in $\Omega$ if
	\begin{equation}\label{eq-supersolution}
		\mathcal{E}(u, \varphi) \geq 0
	\end{equation}
	for all nonnegative functions $\varphi \in C_c^\infty(\Omega)$. A function $u$ is a \textit{(weak) subsolution} if $-u$ is a supersolution, and $u$ is a \textit{(weak) solution} if it is both subsolution and supersolution.
\end{definition}

In \Cref{def-supersolution}, test functions are required to be in $C^\infty_c(\Omega)$, but a standard approximation argument shows the following result.

\begin{lemma}
    Assume that $G$ and $g$ satisfy \eqref{eq-pq}. If $u \in V^{s, G}(\Omega)$ is a supersolution of $\mathcal{L}u=0$ in $\Omega$, then \eqref{eq-supersolution} holds for all nonnegative functions $\varphi \in V^{s, G}_0(\Omega)$.
\end{lemma}

The following theorem shows the solvability of the nonlocal Dirichlet problem for the operator $\mathcal{L}$.

\begin{theorem}\label{thm-DP}
Assume that $G$ and $g$ satisfy \eqref{eq-pq}. Let $\Omega \subset \mathbb{R}^n$ be open and bounded. Let $f \in (V^{s, G}_0(\Omega))^\ast$ and $\vartheta \in V^{s, G}(\Omega)$. There exists a unique solution $u \in V^{s, G}(\Omega)$ of $\mathcal{L}u=f$ in $\Omega$ with $u-\vartheta \in V^{s, G}_0(\Omega)$.
\end{theorem}

\Cref{thm-DP} with $f=0$ follows from Theorem~4.9 in \cite{KL23} by setting $\psi \equiv -\infty$. The general case with $f$ can be obtained by modifying the proof of \cite[Theorem~4.9]{KL23}. Indeed, one can consider the mapping $\mathcal{B}: \mathcal{K}_{\psi, \vartheta}(\Omega) \to (V^{s, G}_0(\Omega))^\ast$ defined by $\langle \mathcal{B}u, w \rangle=\mathcal{E}(u, w)-\int_\Omega fw \,\mathrm{d}x$ instead of $\mathcal{A}$ and follow the lines of the proof.

Let us next provide the definition of a superharmonic function, a subharmonic function, and a harmonic function.

\begin{definition}\label{def-superharmonic}
	A measurable function $u: \mathbb{R}^n \to [-\infty, \infty]$ is a ($\mathcal{L}$-)\emph{superharmonic function} in $\Omega$ if it satisfies the following:
	\begin{enumerate}[(i)]
		\item $u < \infty$ a.e.\ in $\mathbb{R}^n$ and $u>-\infty$ everywhere in $\Omega$,
		\item $u$ is lower semicontinuous in $\Omega$,
		\item if $D \Subset \Omega$ is an open set and $v: \mathbb{R}^n \to [-\infty, \infty]$ is a solution of $\mathcal{L}v=0$ in $D$ with $v \in C(\overline{D})$ and $v_+ \in L^{\infty}(\mathbb{R}^n)$ such that $u \geq v$ on $\partial D$ and almost everywhere on $\mathbb{R}^n \setminus D$, then $u \geq v$ in $D$,
		\item $u_- \in L_{s}^{g}(\mathbb{R}^n)$.
	\end{enumerate}
A function $u$ is said to be ($\mathcal{L}$-)\emph{subharmonic in $\Omega$} if $-u$ is superharmonic in $\Omega$.
\end{definition}

\begin{definition}\label{def-harmonic}
A function $u$ is said to be ($\mathcal{L}$-)\emph{harmonic in $\Omega$} if it is a solution of $\mathcal{L}u=0$ in $\Omega$ which is continuous in $\Omega$.
\end{definition}

We finally present some properties of supersolutions and superharmonic functions obtained in \cite{KL23}.

\begin{theorem}[{\cite[Theorem~5.1]{KL23}}]\label{thm-lsc}
Assume that $G$ and $g$ satisfy \eqref{eq-pq}. Let $u$ be a supersolution of $\mathcal{L}u=0$ in $\Omega$. Then $u(x)=\essliminf_{y \to x}u(y)$ for a.e.\ $x \in \Omega$. In particular, $u$ has a representative that is lower semicontinuous in $\Omega$.
\end{theorem}

\begin{theorem}[{\cite[Theorem~6.1]{KL23}}]\label{thm-relation}
	Assume that $G$ and $g$ satisfy \eqref{eq-pq}.
    \begin{enumerate}[(i)]
		\item
		Let $u$ be a supersolution of $\mathcal{L}u=0$ in $\Omega$ that is lower semicontinuous in $\Omega$ and satisfies
		\begin{equation*}
			u(x)=\essliminf_{y \to x}u(y) \quad \text{for every}~ x \in \Omega.
		\end{equation*}
		Then $u$ is superharmonic in $\Omega$.
		\item
		Assume in addition that $g$ satisfies \eqref{eq-q}. Let $u$ be a superharmonic function in $\Omega$. If $u$ is locally bounded from above in $\Omega$ or $W^{s, G}_{\mathrm{loc}}(\Omega)$, then $u$ is a supersolution of $\mathcal{L}u=0$ in $\Omega$.
	\end{enumerate}
\end{theorem}

Note that \Cref{thm-relation} implies that $u$ is harmonic in $\Omega$ if and only if it is both superharmonic and subharmonic in $\Omega$.

\begin{theorem}[{\cite[Theorem~6.5]{KL23}}]\label{thm-pointwise}
    Assume that $G$ and $g$ satisfy \eqref{eq-pq} and \eqref{eq-q}. Let $u$ be superharmonic in $\Omega$, then
	\begin{equation*}
		u(x)=\liminf_{y \to x}u(y)=\essliminf_{y \to x}u(y) \quad \text{for every}~ x\in \Omega.
	\end{equation*}
	In particular, $\inf_Du=\essinf_Du$ for any open set $D \Subset \Omega$.
\end{theorem}

\subsection{Local behavior of supersolutions}

In this section, we provide local estimates up to the boundary for solutions; more precisely, the local boundedness for subsolutions and the weak Harnack inequality for supersolutions. Such results, which are key ingredients for the proof of both Wolff potential estimates and the Wiener criterion, were developed in \cite{KL23} based on Moser's iteration technique.

\begin{theorem}[{\cite[Theorem~3.3]{KL23}}]\label{thm-loc-bdd}
	Assume that $G$ and $g$ satisfy \eqref{eq-pq}. Let $\Omega \subset \mathbb{R}^n$ be open and let $B_R=B_R(x_0)$ satisfy $B_R \cap \Omega \neq \emptyset$. Let $\varepsilon, p_0 > 0$. If $u \in V^{s, G}(\Omega)$ is a subsolution of $\mathcal{L}u=0$ in $\Omega$, then
	\begin{equation*}
		\esssup_{B_{R/2}} u_M^+ \leq \varepsilon \,\mathrm{Tail}_g(u_M^+; x_0, R/2) + C \left( \fint_{B_R} (u_M^+)^{p_0} \,\mathrm{d}x \right)^{1/p_0},
	\end{equation*}
	where
	\begin{equation*}
		M=\esssup_{B_R \setminus \Omega} u_+, \quad u_M^+(x) = \max\{u_+(x), M\},
	\end{equation*}
	and $C=C(n, p, p_0, q, s, \Lambda, \varepsilon)>0$.
\end{theorem}

\begin{theorem}[{\cite[Theorem~3.4]{KL23}}]\label{thm-WHI}
	Assume that $G$ and $g$ satisfy \eqref{eq-pq}. Let $\tau_1, \tau_2 \in (0,1)$. Let $\delta \in (0, \frac{n}{n-sp})$ if $sp<n$ and $\delta \in (0,\infty)$ if $sp \geq n$. Let $\Omega \subset \mathbb{R}^n$ be open and let $B_R=B_R(x_0)$ satisfy $B_R \cap \Omega \neq \emptyset$. If $u \in V^{s, G}(\Omega)$ is a supersolution of $\mathcal{L}u=0$ in $\Omega$ such that $u\geq0$ in $B_R$, then
	\begin{equation*}
		\fint_{B_{\tau_1R}} g^\delta \left( \frac{u_m^-}{R^s} \right) \,\mathrm{d}x \leq C g^\delta\left( \essinf_{B_{\tau_2R}} \frac{u_m^-}{R^s} \right) + C g^\delta \left( \frac{\mathrm{Tail}_g((u_m^-)_-; x_0, R)}{R^s} \right),
	\end{equation*}
	where
	\begin{equation*}
		m=\essinf_{B_R \setminus \Omega} u, \quad u_m^-(x) = \min\{u(x), m\},
	\end{equation*}
	and $C = C(n, p, q, s, \Lambda, \tau_1, \tau_2, \delta)>0$.
\end{theorem}

The choice $\Omega'=\mathbb{R}^n$ in \cite[Theorem~3.3~and~3.4]{KL23} give \Cref{thm-loc-bdd} and \Cref{thm-WHI}. Note that \Cref{thm-loc-bdd} and \Cref{thm-WHI} cover interior estimates as well.

\section{Potential upper bound}\label{sec-upper}

This section is devoted to the proof of \Cref{thm-Wolff-upper}. Let us first recall the following result in Kilpel\"ainen--Mal\'y~\cite{KM92}, which displays the relation between $p$-superharmonic functions and nonnegative Borel measures in the local setting ($s=1$).

\begin{theorem}[{\cite[Theorem~2.1]{KM92}}]\label{thm-km92}
Let $p>1$. Suppose that $u$ is $p$-superharmonic in $\Omega$. Then there exists a unique nonnegative Borel measure $\mu$ in $\Omega$ such that
\begin{equation*}
\int_{\Omega} |Du|^{p-2} Du \cdot \nabla \varphi \,\mathrm{d}x = \int_{\Omega} \varphi \,\mathrm{d}\mu \quad\text{for every }\varphi \in C^\infty_c(\Omega),
\end{equation*}
where $Du$ is the weak gradient of $u$ defined by
\begin{equation*}
Du=\lim_{j\to \infty} \nabla (u \land j).
\end{equation*}
\end{theorem}

The distributional derivative $\nabla u$ of a $p$-superharmonic function does not necessarily exist when $1<p\leq 2-1/n$, and therefore the equation $-\Delta_pu=-\mathrm{div}(|\nabla u|^{p-2}\nabla u)=\mu$ cannot be understood in the distributional sense in general. This is why the weak gradient $Du$ was introduced in \Cref{thm-km92}.

We now provide the nonlocal counterpart of \Cref{thm-km92} in the Orlicz growth framework. It is noteworthy that, in the theorem below, a superharmonic function $u$ in $\Omega$ solves $\mathcal{L}u=\mu$ in $\Omega$ in the sense of distribution as the equation no longer needs gradients to be defined.

\begin{theorem}\label{thm-measure}
Let $p>1$. Suppose that $u$ is superharmonic in $\Omega$. Then there exists a unique nonnegative Borel measure $\mu$ in $\Omega$ such that $\mathcal{L}u=\mu$ in $\Omega$ in the sense of distribution, i.e.\
\begin{equation*}
\mathcal{E}(u, \varphi) = \int_{\Omega} \varphi \,\mathrm{d}\mu \quad \text{for every $\varphi \in C^\infty_c(\Omega)$}.
\end{equation*}
\end{theorem}

The measure $\mu$ given in \Cref{thm-measure} is called the \emph{Riesz measure}.

\begin{proof}
Let $\varphi \in C_c^\infty(\Omega)$ be nonnegative and let $D \Subset D' \Subset \Omega$ be such that $\mathrm{supp}\,\varphi \subset D$. Moreover, we define $u_j=u \land j$ for each $j \in \mathbb{N}$. Then since $u_j$ is a supersolution of $\mathcal{L}u_j=0$ in $\Omega$ by \Cref{thm-relation}~(ii), we observe that $\mathcal{E}(u_j, \varphi)$ is well defined and nonnegative.

We now claim that we can pass to the limit, i.e.
\begin{equation}\label{eq-claim-passlimit}
    \mathcal{E}(u, \varphi)=\lim_{j \to \infty}\mathcal{E}(u_j, \varphi).
\end{equation}
For this purpose, we first note that superharmonic functions do not belong to $W^{s, G}_{\mathrm{loc}}(\Omega)$ in general, but they have some integrability. Indeed, \cite[Theorem~6.6]{KL23} shows that $u \in L^{g^\delta}_{\mathrm{loc}}(\Omega) \cap W^{\sigma, g^\alpha}_{\mathrm{loc}}(\Omega)$ for any
\begin{equation*}
\delta \in 
\begin{cases}
(0, \frac{n}{n-sp}) &\text{if}~ sp<n, \\
(0, \infty) &\text{if}~ sp \geq n,
\end{cases}
\quad \sigma \in (0, s), \quad\text{and}\quad \alpha \in (0, \min\{\tfrac{n}{n-sp/q}, \tfrac{q}{q-1} \}).
\end{equation*}
We fix some $\alpha \in (1, \min\{\tfrac{n}{n-sp/q}, \tfrac{q}{q-1} \})$. Then functions $v_j:=g(|D^su_j|)\frac{D^su_j}{|D^su_j|} |x-y|^{(s-\sigma)(q-1)}$ form a bounded sequence in $L^\alpha(D' \times D', \frac{k(x, y)}{|x-y|^n}\,\mathrm{d}y\,\mathrm{d}x)$ since
\begin{equation*}
\iint_{D' \times D'} |v_j|^\alpha \frac{k(x, y)}{|x-y|^n} \,\mathrm{d}y \,\mathrm{d}x \leq C \iint_{D' \times D'} g^\alpha(|D^\sigma u_j|) \frac{\mathrm{d}y \,\mathrm{d}x}{|x-y|^n} \leq C \varrho_{W^{\sigma, g^\alpha}(D')}(u) \leq C,
\end{equation*}
where $C$ is independent of $j$. This together with the convergence
\begin{equation*}
v_j \to v:= g(|D^su|)\frac{D^su}{|D^su|} |x-y|^{(s-\sigma)(q-1)} \quad\text{a.e.\ in }D' \times D'
\end{equation*}
shows that $v_j$ converges weakly to $v$ in $L^\alpha(D' \times D', \frac{k(x, y)}{|x-y|^n}\,\mathrm{d}y\,\mathrm{d}x)$ as $j \to \infty$. If we take $\sigma \in (0, s)$ sufficiently close to $s$ so that $s+(s-\sigma)(q-1)<1$, then we have
\begin{equation*}
\iint_{D' \times D'} \left( |D^s\varphi| |x-y|^{-(s-\sigma)(q-1)} \right)^{\frac{\alpha}{\alpha-1}} \frac{k(x, y)}{|x-y|^n} \,\mathrm{d}y \,\mathrm{d}x \leq C \|\varphi\|_{W^{s+(s-\sigma)(q-1), \frac{\alpha}{\alpha-1}}(D')}^{\frac{\alpha}{\alpha-1}} < \infty,
\end{equation*}
and hence
\begin{align}\label{eq-DD}
\begin{split}
&\int_{D'} \int_{D'} g(|D^su|) \frac{D^su}{|D^su|} D^s\varphi \frac{k(x, y)}{|x-y|^n} \,\mathrm{d}y\,\mathrm{d}x \\
&= \lim_{j \to \infty} \int_{D'} \int_{D'} g(|D^su_j|) \frac{D^su_j}{|D^su_j|} D^s\varphi \frac{k(x, y)}{|x-y|^n} \,\mathrm{d}y\,\mathrm{d}x.
\end{split}
\end{align}
On the other hand, we observe that
\begin{equation*}
\left| g(|D^su_j|) \frac{D^su_j}{|D^su_j|} \frac{\varphi(x)}{|x-y|^s} \right| \leq \left( g\left( \frac{|u_j(x)|}{|x-y|^s} \right) + g\left( \frac{|u_j(y)|}{|x-y|^s} \right) \right) \frac{\varphi(x)}{|x-y|^s}.
\end{equation*}
Since
\begin{align*}
\int_{D} \int_{\mathbb{R}^n \setminus D'} g\left( \frac{|u_j(x)|}{|x-y|^s} \right) \varphi(x) \frac{k(x, y)}{|x-y|^{n+s}} \,\mathrm{d}y\,\mathrm{d}x
&\leq C \int_D g(|u(x)|) \,\mathrm{d}x \quad\text{and} \\
\int_{D} \int_{\mathbb{R}^n \setminus D'} g\left( \frac{|u_j(y)|}{|x-y|^s} \right) \varphi(x) \frac{k(x, y)}{|x-y|^{n+s}} \,\mathrm{d}y\,\mathrm{d}x
&\leq C \int_{\mathbb{R}^n} g\left( \frac{|u(y)|}{(1+|y|)^s} \right) \frac{\mathrm{d}y}{(1+|y|)^{n+s}},
\end{align*}
the dominated convergence theorem yields
\begin{align}\label{eq-DDc}
\begin{split}
&\int_{D'} \int_{\mathbb{R}^n \setminus D'} g(|D^su|) \frac{D^su}{|D^su|} D^s\varphi \frac{k(x, y)}{|x-y|^n} \,\mathrm{d}y\,\mathrm{d}x \\
&= \lim_{j \to \infty} \int_{D'} \int_{\mathbb{R}^n \setminus D'} g(|D^su_j|) \frac{D^su_j}{|D^su_j|} D^s\varphi \frac{k(x, y)}{|x-y|^n} \,\mathrm{d}y\,\mathrm{d}x.
\end{split}
\end{align}
Thus, we deduce from \eqref{eq-DD} and \eqref{eq-DDc} that the claim \eqref{eq-claim-passlimit} holds. It implies that the assignment $\varphi \mapsto \mathcal{E}(u, \varphi)$ is a nonnegative distribution. Therefore, it can be represented by a nonnegative Borel measure $\mu$.
\end{proof}

We next provide a pointwise upper bound of superharmonic functions in terms of the Wolff potential. We modify the methods used by Kilpel\"ainen--Mal\'y~\cite{KM94} and Lukkari--Maeda--Marola~\cite{LMM10} in proving potential upper bounds for local equations. The following estimate can be understood as a De Giorgi-type reverse H\"older inequality with measure data. Note that $p_{\sigma}^{\ast}$ denotes the fractional Sobolev exponent $p_{\sigma}^\ast:=np/(n-\sigma p)$.

\begin{lemma}\label{lem-DG}
Let $\sigma \in (0, \min\{s, n/p\})$ and assume
\begin{equation}\label{eq-delta}
1<\delta< \frac{(p-1)^{\ast}_{\sigma}}{p-1} = \frac{n}{n-\sigma(p-1)}.
\end{equation}
Let $u$ be a superharmonic function in $B_{R}=B_{R}(x_{0})$ and let $\mu=\mathcal{L}u$ be the Riesz measure. Then there exists a constant $C_{0}=C_0(n, p, q, s, \sigma, \Lambda, \delta) > 0$ such that
\begin{align}\label{eq-DG-claim}
	\begin{split}
		\left( \fint_{B_{R/2}} \left( \frac{\bar{g}(v)}{\bar{g}(\lambda)} \right)^{\delta} \,\mathrm{d}x \right)^{p/p^{\ast}_{\sigma}}
		&\leq C_{0} \left( \theta_{R}^{p/p^{\ast}_{\sigma}} + \fint_{B_{R}} \left( \frac{\bar{g}(v)}{\bar{g}(\lambda)} \right)^{\delta} \,\mathrm{d}x \right) \\
		&\quad + \frac{C_{0}}{\bar{g}(\lambda)} \left( \frac{\mu(B_{R})}{R^{n-s}} + T(u_+; x_0, R) \right)
	\end{split}
\end{align}
for any $\lambda > 0$, where $v=u_{+}/R^{s}$.

Moreover, there exists a constant $C=C(n, p, q, s, \sigma, \Lambda, \delta) > 0$ such that
\begin{align}\label{eq-DG}
\begin{split}
\left( \fint_{B_{R/2}} \bar{g}^{\delta}\left( \frac{u_{+}}{(R/2)^{s}} \right) \,\mathrm{d}x \right)^{1/\delta}
&\leq C \theta_{R}^{\rho_{1}} \left( \fint_{B_{R}} \bar{g}^{\delta}\left( \frac{u_{+}}{R^{s}} \right) \,\mathrm{d}x \right)^{1/\delta} \\
&\quad + C\theta_{R}^{\rho_{2}} \left( \frac{\mu(B_{R})}{R^{n-s}} + T(u_+; x_0, R) \right),
\end{split}
\end{align}
where $\rho_1=\sigma p/(n\delta)$, $\rho_2=1/\delta-p/p^{\ast}_{\sigma}$, and $\theta_{R} = |B_{R} \cap \lbrace u>0 \rbrace| / |B_{R}|$.
\end{lemma}

\begin{proof}
Since $p < n/\sigma$, we have $p^{\ast}_{\sigma}=np/(n-\sigma p) < \infty$. We first prove the estimate \eqref{eq-DG-claim}. For
\begin{equation*}
\tilde{\delta} = \frac{p(q-1)\delta}{q-\delta},
\end{equation*}
it is easy to check that $p<\tilde{\delta}<p^{\ast}_{\sigma}$. We define a nonnegative function
\begin{equation*}
w= \left( \frac{\bar{g}(\lambda+v)}{\bar{g}(\lambda)} \right)^{\delta/\tilde{\delta}}-1.
\end{equation*}
Note that we have
\begin{align*}
\left( \frac{\bar{g}(v)}{\bar{g}(\lambda)} \right)^{\delta}
&= \left( \frac{\bar{g}(v)}{\bar{g}(\lambda)} \right)^{\delta} {\bf 1}_{\lbrace 0<v<\lambda \rbrace} + \left( \frac{\bar{g}(v)}{\bar{g}(\lambda)} \right)^{\delta} {\bf 1}_{\lbrace v\geq \lambda \rbrace} \\
&\leq {\bf 1}_{\lbrace u>0 \rbrace} + \left( \frac{\bar{g}(v)}{\bar{g}(\lambda)} \right)^{\delta p^{\ast}_{\sigma}/\tilde{\delta}} {\bf 1}_{\lbrace v\geq \lambda \rbrace} \\
&\leq {\bf 1}_{\lbrace u>0 \rbrace} + (w+1)^{p^{\ast}_{\sigma}} {\bf 1}_{\lbrace u > 0 \rbrace} \\
&\leq C {\bf 1}_{\lbrace u>0 \rbrace} + C w^{p^{\ast}_{\sigma}},
\end{align*}
and hence
\begin{equation}\label{eq-DG-w}
\fint_{B_{R/2}} \left( \frac{\bar{g}(v)}{\bar{g}(\lambda)} \right)^{\delta} \,\mathrm{d}x \leq C\theta_{R} + C \fint_{B_{R/2}} w^{p^{\ast}_{\sigma}} \,\mathrm{d}x
\end{equation}
for some $C = C(n, p, \sigma) > 0$.

We apply the fractional Sobolev inequality to have
\begin{equation}\label{eq-DG-I1I2}
\left( \fint_{B_{R/2}} w^{p^{\ast}_{\sigma}} \,\mathrm{d}x \right)^{p/p^{\ast}_{\sigma}} \leq C \fint_{B_{R/2}} w^{p} \,\mathrm{d}x + C R^{\sigma p} \fint_{B_{R/2}} \int_{B_{R/2}} |D^\sigma w|^p \,\mathrm{d}y \,\mathrm{d}x =: I_{1} + I_{2}.
\end{equation}
Since $w \leq (\bar{g}(2\lambda)/\bar{g}(\lambda))^{\delta/\tilde{\delta}} \leq C$ on $\lbrace 0<v<\lambda \rbrace$ and $w \leq C(\bar{g}(v)/\bar{g}(\lambda))^{\delta/\tilde{\delta}}$ on $\lbrace v \geq \lambda \rbrace$ by \Cref{lem-G}, we obtain
\begin{equation}\label{eq-DG-I1}
I_{1} \leq C \theta_{R} + C \fint_{B_{R/2}} \left( \frac{\bar{g}(v)}{\bar{g}(\lambda)} \right)^{\delta} \,\mathrm{d}x.
\end{equation}
For $I_{2}$, we apply \Cref{lem-ineq1} with $\alpha=\delta/\tilde{\delta}$ to have
\begin{align*}
I_{2}
&\leq C \fint_{B_{R/2}} \int_{B_{R/2}} \left( \frac{|x-y|}{R} \right)^{(s-\sigma)p} \left( \fint_{v(y)}^{v(x)} \left( \frac{\bar{g}(\lambda+t)}{\bar{g}(\lambda)} \right)^{\delta p/\tilde{\delta}} \,\mathrm{d}t \right) {\bf 1}_{\{v(x) \neq v(y)\}}(x, y) \frac{\mathrm{d}y \,\mathrm{d}x}{|x-y|^{n}} \\
&\quad + C \fint_{B_{R/2}} \int_{B_{R/2}} G(|D^su_+|) \left( \fint_{v(y)}^{v(x)} \left( \frac{\bar{g}(\lambda+t)}{\bar{g}(\lambda)} \right)^{\delta p/\tilde{\delta}} \frac{\mathrm{d}t}{G(\lambda +t)} \right) \frac{\mathrm{d}y \,\mathrm{d}x}{|x-y|^{n}} \\
&=: I_{2,1} + I_{2,2}.
\end{align*}
Since ${\bf 1}_{\lbrace v(x) > v(y) \rbrace}(x, y) \leq {\bf 1}_{\lbrace u(x) >0 \rbrace}(x)$, we obtain
\begin{align}\label{eq-DG-I21}
\begin{split}
I_{2,1}
&\leq C \fint_{B_{R/2}} \left( \frac{\bar{g}(\lambda +v(x))}{\bar{g}(\lambda)} \right)^{\delta p/\tilde{\delta}} {\bf 1}_{\lbrace u(x) > 0 \rbrace}(x) \int_{B_{R/2}} \left( \frac{|x-y|}{R} \right)^{(s-\sigma)p} \frac{\mathrm{d}y \,\mathrm{d}x}{|x-y|^{n}} \\
&\leq C \fint_{B_{R/2}} \left( \frac{\bar{g}(\lambda +v(x))}{\bar{g}(\lambda)} \right)^{\delta p/\tilde{\delta}} {\bf 1}_{\lbrace u(x) > 0 \rbrace}(x) \,\mathrm{d}x \\
&\leq C \theta_{R} + C \fint_{B_{R/2}} \left( \frac{\bar{g}(v)}{\bar{g}(\lambda)} \right)^{\delta} \,\mathrm{d}x.
\end{split}
\end{align}
To estimate $I_{2,2}$, we define an auxiliary function
\begin{equation*}
\varphi = 1-\left( \frac{\bar{g}(\lambda)}{\bar{g}(\lambda+v)} \right)^{\frac{\delta-1}{q-1}}
\end{equation*}
and let $\eta \in C^{\infty}_{c}(B_{3R/4})$ be a cut-off function satisfying $0\leq \eta \leq 1$, $\eta \equiv 1$ on $B_{R/2}$, and $|\nabla \eta| \leq C/R$. Since $\delta p/\tilde{\delta}-1 = -(\delta-1)/(q-1)$, we have
\begin{equation*}
I_{2,2} \leq \frac{C}{\bar{g}(\lambda)} \fint_{B_{R}} \int_{B_{R}} G(|D^su_+|) \left( \fint_{v(y)}^{v(x)} \left( \frac{\bar{g}(\lambda)}{\bar{g}(\lambda+t)} \right)^{\frac{\delta-1}{q-1}} \frac{\mathrm{d}t}{\lambda +t} \right) (\eta(x) \land \eta(y))^{q} \frac{\mathrm{d}y \,\mathrm{d}x}{|x-y|^{n}}.
\end{equation*}
By using $\varphi\eta^{q}$ as a test function in $\mathcal{L}u=\mu$ and applying \Cref{lem-ineq2}, we obtain
\begin{align*}
I_{2,2}
&\leq C \fint_{B_{R}} \left( \frac{\bar{g}(\lambda+v(x))}{\bar{g}(\lambda)} \right)^{\delta} {\bf 1}_{\lbrace u(x) > 0 \rbrace}(x) \,\mathrm{d}x + \frac{C}{\bar{g}(\lambda)} \frac{\mu(B_{R})}{R^{n-s}}\\
&\quad -C\frac{R^{s}}{\bar{g}(\lambda)} \fint_{B_R} \int_{\mathbb{R}^n \setminus B_R} g(|D^su|) \frac{D^su}{|D^su|} \frac{\varphi(x)\eta^q(x)}{|x-y|^s} \frac{\mathrm{d}y \,\mathrm{d}x}{|x-y|^n} \\
&\leq C \theta_{R} + C \fint_{B_{R}} \left( \frac{\bar{g}(v)}{\bar{g}(\lambda)} \right)^{\delta} \,\mathrm{d}x + \frac{C}{\bar{g}(\lambda)} \frac{\mu(B_{R})}{R^{n-s}}\\
&\quad + C\frac{R^{s}}{\bar{g}(\lambda)} \fint_{B_{3R/4}} \int_{\mathbb{R}^n \setminus B_R} g\left(\frac{u_{+}(y)}{|x-y|^s}\right) {\bf 1}_{\lbrace u(x) > 0 \rbrace}(x) \frac{\mathrm{d}y \,\mathrm{d}x}{|x-y|^{n+s}}.
\end{align*}
Since $|y-x_{0}| \leq |y-x|+3R/4 \leq 4|x-y|$ for $x \in B_{3R/4}(x_0)$ and $y \in \mathbb{R}^n \setminus B_R(x_0)$, we have
\begin{equation}\label{eq-DG-I22}
I_{2, 2} \leq C \theta_{R} + C \fint_{B_{R}} \left( \frac{\bar{g}(v)}{\bar{g}(\lambda)} \right)^{\delta} \,\mathrm{d}x + \frac{C}{\bar{g}(\lambda)} \left( \frac{\mu(B_{R})}{R^{n-s}} + T(u_+; x_0, R) \right).
\end{equation}
Therefore, the claim \eqref{eq-DG-claim} follows by combining \eqref{eq-DG-w}--\eqref{eq-DG-I22} and using $\theta_R \leq \theta_R^{p/p^{\ast}_{\sigma}}$.

Let us next deduce \eqref{eq-DG} from \eqref{eq-DG-claim}. We set 
\begin{align*}
	A_R=\left(\fint_{B_R} \bar{g}^{\delta}\left(\frac{u_+}{R^s}\right)\,\mathrm{d}x\right)^{1/\delta}
\end{align*}
and $M=1+C_0$. If we choose $\lambda$ so that
\begin{equation*}
	\frac{1}{\bar{g}(\lambda)}=2^{s(q-1)}M^{p^{\ast}_{\sigma}/(\delta p)}\theta_R^{1/\delta}A_{R/2}^{-1},
\end{equation*}
then it follows from \Cref{lem-G} and \eqref{eq-DG-claim} that
\begin{align*}
	\theta_R^{p/p^{\ast}_{\sigma}}&=\left(2^{s(1-q)}\frac{A_{R/2}}{\bar{g}(\lambda)}\right)^{\delta p/p^{\ast}_{\sigma}}-C_0\theta_R^{p/p^{\ast}_{\sigma}} \\
	&\leq \left(\fint_{B_{R/2}} \left(\frac{\bar{g}(v)}{\bar{g}(\lambda)}\right)^{\delta} \, \mathrm{d}x\right)^{p/p^{\ast}_{\sigma}}-C_0\theta_R^{p/p^{\ast}_{\sigma}}\\
	&\leq C_{0} \left( \left(\frac{A_R}{\bar{g}(\lambda)}\right)^{\delta}+ \frac{1}{\bar{g}(\lambda)} \frac{\mu(B_{R})}{R^{n-s}} +\frac{T(u_+; x_0, R)}{\bar{g}(\lambda)}  \right)\\
	&\leq C_1 \left(\theta_R \left(\frac{A_R}{A_{R/2}}\right)^{\delta}+\theta_R^{1/\delta}A_{R/2}^{-1}\frac{\mu(B_R)}{R^{n-s}}+\theta_R^{1/\delta}A_{R/2}^{-1} T(u_+; x_0, R) \right).
\end{align*}
This implies that
\begin{equation*}
\frac{1}{3}\theta_R^{p/p^{\ast}_{\sigma}} \leq \left( C_{1} \theta_R \left(\frac{A_R}{A_{R/2}}\right)^{\delta} \right) \lor \left( C_{1} \theta_R^{1/\delta}A_{R/2}^{-1}\frac{\mu(B_R)}{R^{n-s}} \right) \lor \left( C_{1} \theta_R^{1/\delta}A_{R/2}^{-1} T(u_+; x_0, R) \right).
\end{equation*}
In other words, we have
\begin{equation*}
A_{R/2} \leq \left( (3C_1)^{1/\delta} \theta_R^{\rho_1}A_R \right) \lor \left( 3C_1 \theta_R^{\rho_2} \frac{\mu(B_R)}{R^{n-s}} \right) \lor \left( 3C_1 \theta_R^{\rho_2} T(u_+; x_0, R)\right),
\end{equation*}
where $\rho_1=(1-p/p^{\ast}_{\sigma})/\delta = \sigma p/(n\delta)>0$ and $\rho_2=1/\delta-p/p^{\ast}_{\sigma}>0$. Hence, the desired estimate \eqref{eq-DG} follows.
\end{proof}

We prove \Cref{thm-Wolff-upper} by iterating \Cref{lem-DG} and using the weak Harnack inequality (\Cref{thm-WHI}).

\begin{proof}[Proof of \Cref{thm-Wolff-upper}]
By \Cref{thm-measure} there exists a unique nonnegative Borel measure $\mu$ in $\Omega$ such that $\mathcal{L}u=\mu$ in $\Omega$ in the sense of distribution.

	For $j=0, 1, \cdots$, we set $R_j=2^{-j-1}R$ and $B^j=B_{R_j}$. Let $\sigma \in (0,\min\{s, n/p\})$ and let $\delta$ be any constant satisfying \eqref{eq-delta}. For $\varepsilon \in (0,1)$ to be determined later, we define $l_0=0$ and 
	\begin{equation*}
		l_{j+1}=l_j+R_j^s\bar{g}^{-1}\left(\frac{1}{\varepsilon}\left(\fint_{B^j}\bar{g}^{\delta}\left(\frac{(u-l_j)_+}{R_j^s}\right)\,\mathrm{d}x\right)^{1/\delta}\right) \quad \text{for}~ j=0,1, \cdots,
	\end{equation*}
and set $\theta_j=|B^{j-1}\cap \{u>l_j\}|/|B^{j-1}|$ for $j=1,2, \cdots.$ Note that we have
\begin{align}\label{eq-theta-cond}
	\begin{split}
	\theta_j &\leq \frac{1}{|B^{j-1}|} \int_{B^{j-1} \cap \{u>l_j\}} \left(\frac{\bar{g}((u-l_{j-1})_+/R^s_{j-1})}{\bar{g}((l_j-l_{j-1})/R^s_{j-1})}\right)^{\delta} \,\mathrm{d}x \\
	&\leq \frac{1}{|B^{j-1}|} \int_{B^{j-1}} \left(\frac{\bar{g}((u-l_{j-1})_+/R^s_{j-1})}{\bar{g}((l_j-l_{j-1})/R^s_{j-1})}\right)^{\delta} \,\mathrm{d}x=\varepsilon^{\delta}.
	\end{split}
\end{align}
By applying \Cref{lem-DG} with $u-l_j$ in $B^{j-1}$ for $j \geq 1$ and using \eqref{eq-theta-cond}, we obtain
\begin{align*}
\bar{g}\left( \frac{l_{j+1}-l_{j}}{R_{j}^{s}} \right) 
&\leq C \varepsilon^{\rho_{1}\delta} \bar{g}\left( \frac{l_{j}-l_{j-1}}{R_{j-1}^{s}} \right) + C \varepsilon^{\rho_{2}\delta - 1} \left( \frac{\mu(B^{j-1})}{R_{j-1}^{n-s}} + T((u-l_j)_+; x_0, R_{j-1}) \right) \\
&\leq C \varepsilon^{\rho_{1}\delta} \bar{g}\left( \frac{l_{j}-l_{j-1}}{R_{j-1}^{s}} \right) + C \varepsilon^{\rho_{2}\delta - 1} \left( \frac{\mu(B^{j-1})}{R_{j-1}^{n-s}} + 2^{-sj} T(u_+; x_0,  R_0)+T_{j} \right),
\end{align*}
where
\begin{equation*}
T_{j} = R_{j-1}^s \int_{B^0 \setminus B^{j-1}} g\left(\frac{(u(x)-l_j)_+}{|x-x_0|^s}\right) \frac{\mathrm{d}x}{|x-x_0|^{n+s}}.
\end{equation*}
Note that we have
\begin{align*}
T_{j}
&\leq \sum_{i=2}^{j} \frac{R_{j-1}^{s}}{R_{j-i+1}^{n+s}} \int_{B^{j-i} \setminus B^{j-i+1}} g\left(\frac{(u(x)-l_j)_+}{R_{j-i+1}^s}\right) \,\mathrm{d}x \\
&\leq C \sum_{i=2}^{j} 2^{-si} \fint_{B^{j-i}} g\left(\frac{(u(x)-l_j)_+}{R_{j-i}^s}\right) \,\mathrm{d}x \\
&\leq C \sum_{i=2}^{j} 2^{-si} \left(\fint_{B^{j-i}} g^{\delta}\left(\frac{(u(x)-l_j)_+}{R_{j-i}^{s}}\right) \mathrm{d}x\right)^{1/\delta} \left(\frac{|B^{j-i} \cap \{u>l_j\}|}{|B^{j-i}|}\right)^{1-1/\delta} \\
&\leq C \sum_{i=2}^{j} 2^{-si} \varepsilon \bar{g}\left( \frac{l_{j-i+1}-l_{j-i}}{R_{j-i}^{s}} \right) \theta_{j-i+1}^{1-1/\delta} \\
&\leq C \varepsilon^{\delta} \sum_{i=2}^{j} 2^{-si} g\left( \frac{l_{j-i+1}-l_{j-i}}{R_{j-i}^{s}} \right)
\end{align*}
by using H\"older's inequality, \eqref{eq-bar-g-comp}, and \eqref{eq-theta-cond}. We apply \Cref{lem-g-dyda}~(ii) with $\beta=2^{\frac{s}{2(q-1)}}$ and use \Cref{lem-G} to deduce
\begin{align*}
T_{j}
\leq C \varepsilon^{\delta} g\left( \sum_{i=2}^{j} \beta^{i} g^{-1} \left( 2^{-si} g\left( \frac{l_{j-i+1}-l_{j-i}}{R_{j-i}^{s}} \right) \right) \right) \leq C \varepsilon^{\delta} g\left( \sum_{i=2}^{j} 2^{-\frac{s}{2(q-1)}i} \frac{l_{j-i+1}-l_{j-i}}{R_{j-i}^{s}} \right).
\end{align*}
We have obtained
\begin{align*}
\bar{g}\left( \frac{l_{j+1}-l_{j}}{R_{j}^{s}} \right)
&\leq C \varepsilon^{\rho_{1}\delta} \bar{g}\left( \frac{l_{j}-l_{j-1}}{R_{j-1}^{s}} \right) + C \varepsilon^{\rho_{2}\delta-1+\delta} g\left( \sum_{i=2}^{j} 2^{-\frac{s}{2(q-1)}i} \frac{l_{j-i+1}-l_{j-i}}{R_{j-i}^{s}} \right) \\
&\quad + C \varepsilon^{\rho_{2}\delta - 1} \left( \frac{\mu(B^{j-1})}{R_{j-1}^{n-s}} + 2^{-sj} T(u_+; x_0, R_0) \right).
\end{align*}
Recall that $\rho_{2}\delta-1+\delta=\delta(1-p/p^{\ast}_{\sigma}) > 0$. By taking $\bar{g}^{-1}$ on both sides and using \Cref{lem-G}, we have
\begin{align*}
l_{j+1} - l_{j}
&\leq C \varepsilon^{\frac{\rho_1 \delta}{q-1}}(l_{j}-l_{j-1}) + C \varepsilon^{\frac{\rho_{2}\delta-1+\delta}{q-1}} \sum_{i=2}^j 2^{-\frac{s}{2(q-1)}i} (l_{j-i+1}-l_{j-i}) \\
&\quad + C(\varepsilon) \left( R_{j}^{s} \bar{g}^{-1} \left( \frac{\mu(B^{j-1})}{R_{j-1}^{n-s}} \right) + 2^{-\frac{sq}{q-1}j} \mathrm{Tail}_{g}(u_{+}; x_{0}, R_{0}) \right),
\end{align*}
which implies that for $k \geq 2$
\begin{align*}
	l_k&=l_1+\sum_{j=1}^{k-1}(l_{j+1}-l_j)\\
	&\leq l_1+C\varepsilon^{\frac{\rho_1 \delta}{q-1}}l_{k-1}+C\varepsilon^{\frac{\rho_{2}\delta-1+\delta}{q-1}} \sum_{j=1}^{k-1}\sum_{i=2}^j 2^{-\frac{s}{2(q-1)}i} (l_{j-i+1}-l_{j-i})\\
	&\quad + C(\varepsilon) \sum_{j=1}^{k-1} \left( R^s_{j}\bar{g}^{-1}\left(\frac{\mu(B^{j-1})}{R_{j-1}^{n-s}}\right)+ \mathrm{Tail}_g(u_+; x_0, R_{0}) \right).
\end{align*}
Since
\begin{align*}
	 \sum_{j=1}^{k-1}\sum_{i=2}^j 2^{-\frac{s}{2(q-1)}i} (l_{j-i+1}-l_{j-i})= \sum_{i=2}^{k-1} 2^{-\frac{s}{2(q-1)}i}\sum_{j=i}^{k-1} (l_{j-i+1}-l_{j-i}) \leq \sum_{i=2}^{k-1} 2^{-\frac{s}{2(q-1)}i}l_k \leq Cl_k
\end{align*}
and $\sum_{j} 2^{-\frac{sq}{q-1}j} \leq C$, we obtain
\begin{equation*}
l_k \leq l_1+C \left( \varepsilon^{\frac{\rho_1 \delta}{q-1}}+\varepsilon^{\frac{\rho_{2}\delta-1+\delta}{q-1}} \right) l_{k} +C(\varepsilon) \left( {\bf W}^{\mu}_{s, G}(x_{0}, R) + \mathrm{Tail}_g(u_+; x_0, R_{0}) \right).
\end{equation*}
We now take $\varepsilon \in (0,1)$ sufficiently small so that $C(\varepsilon^{\frac{\rho_1 \delta}{q-1}}+\varepsilon^{\frac{\rho_{2}\delta-1+\delta}{q-1}}) <1/2$, then
\begin{equation*}
l_{k} \leq 2l_{1} + C\, {\bf W}^{\mu}_{s, G}(x_{0}, R) + C\, \mathrm{Tail}_g(u_+; x_0, R_{0}).
\end{equation*}
Note that \eqref{eq-theta-cond} implies $\inf_{B^{k-1}}u \leq l_k$ since $\varepsilon^{\delta}<1$. By the lower semicontinuity of $u$, we have
\begin{align*}
	u(x_0)
	&\leq \liminf_{k \to \infty}l_k \\
	&\leq CR^s\bar{g}^{-1}\left(\left(\fint_{B_{R/2}}\bar{g}^{\delta}\left(\frac{u_+}{R^s}\right)\,\mathrm{d}x\right)^{1/\delta}\right)+C\,\mathbf{W}_{s, G}^{\mu}(x_0, R)+C\,\mathrm{Tail}_g(u_+; x_0, R_{0}).
\end{align*}
Finally, since $\delta<n/(n-\sigma(p-1))<n/(n-sp)$, an application of \Cref{thm-WHI} together with \Cref{lem-G} finishes the proof.
\end{proof}

\section{Potential lower bound}\label{sec-lower}

In this section, we prove the lower bound of the Wolff potential estimate. We adapt the method used by Kilpel\"ainen--Mal\'y~\cite{KM92} to nonlocal equations with Orlicz growth.

Let us first provide some lemmas. The first lemma is a Caccioppoli-type estimate, which is well known in the literature. See Byun--Kim--Ok~\cite{BKO23} and Chaker--Kim--Weidner~\cite{CKW22} for instance.

\begin{lemma}\label{lem-Caccioppoli}
Let $\eta \in C^{\infty}_{c}(B_{R})$ be nonnegative and let $u$ be a supersolution of $\mathcal{L}u=0$ in $B_{R}=B_{R}(x_{0})$ such that $u \geq 0$ in $B_{R}$. There exists a constant $C = C(p, q, \Lambda) > 0$ such that for any $l \in \mathbb{R}$,
\begin{align*}
&\int_{B_{R}} \int_{B_{R}} G(|D^sw_-|) (\eta(x) \lor \eta(y))^{q} \frac{\mathrm{d}y \,\mathrm{d}x}{|x-y|^{n}} \\
&\leq C \int_{B_{R}} \int_{B_{R}} G\left( \frac{w_{-}(x) \lor w_{-}(y)}{R^{s}} \right) (R^{s} |D^s\eta|)^{q} \frac{\mathrm{d}y \,\mathrm{d}x}{|x-y|^{n}} \\
&\quad + C \left( \int_{B_{R}} w_{-}(x) \,\mathrm{d}x \right) \sup_{x \in \mathrm{supp}\,\eta} \int_{\mathbb{R}^{n} \setminus B_{R}} g\left( \frac{w_{-}(y)}{|x-y|^{s}} \right) \frac{\mathrm{d}y}{|x-y|^{n+s}},
\end{align*}
where $w = u-l$.
\end{lemma}

The next lemma is a nonlocal analogue of \cite[Lemma~3.4]{KM92}.

\begin{lemma}\label{lem-aux}
Let $u$ be a supersolution of $\mathcal{L}u=0$ in $B_{R}=B_{R}(x_{0})$ such that $u \geq 0$ in $B_{R}$. Let $\eta \in C_{c}^{\infty}(B_{R})$ be a function satisfying $|\nabla \eta| \leq c/R$. Then
\begin{equation}\label{eq-aux}
\fint_{B_{3R/4}} \int_{B_{3R/4}} g(|D^su|) R^{s} |D^s\eta| \frac{\mathrm{d}y\,\mathrm{d}x}{|x-y|^{n}} \leq C g\left( \inf_{B_{R/2}} \frac{u}{R^{s}} \right) + C\, T(u_-; x_0, R)
\end{equation}
for some constant $C=C(n, p, q, s, \Lambda, c) > 0$.
\end{lemma}

\begin{proof}
For
\begin{equation}\label{eq-d}
d=\mathrm{Tail}_{g}(u_{-}; x_{0}, R),
\end{equation}
we define a function $v=(u+d)/R^s$ and set $m=\inf_{B_{R/2}}v$. Let $A$ denote the left-hand side of \eqref{eq-aux} and let
\begin{equation}\label{eq-tau}
0 < \tau < 1 \land \frac{sp}{(q-1)(n-sp)_+}.
\end{equation}
Here, $1/(n-sp)_+$ is understood as $\infty$ when $sp \geq n$. By applying the inequality \eqref{eq-alg} with
\begin{equation*}
a=|D^su|, \quad b= R^{s} |D^s\eta|, \quad \varepsilon = \bar{g}^{\tau}(m) \fint_{v(y)}^{v(x)} \frac{\mathrm{d}t}{t\bar{g}^{\tau}(t)},
\end{equation*}
and using \eqref{eq-pq}, we have
\begin{align*}
A
&\leq q \fint_{B_{3R/4}} \int_{B_{3R/4}} \varepsilon G(|D^su|) \frac{\mathrm{d}y\,\mathrm{d}x}{|x-y|^{n}} + \fint_{B_{3R/4}} \int_{B_{3R/4}} g\left( \frac{1}{\varepsilon} R^{s} |D^s\eta| \right) R^{s} |D^s\eta| \frac{\mathrm{d}y\,\mathrm{d}x}{|x-y|^{n}} \\
&=: I_{1} + I_{2}.
\end{align*}
The Caccioppoli estimate (see \cite[Lemma~3.6~(i)]{KL23} for instance) shows that
\begin{equation*}
I_{1} \leq C \bar{g}^{\tau}(m) \left( 1+\frac{T((u+d)_{-}; x_{0}, 7R/8)}{g(d/(7R/8)^{s})} \right) \fint_{B_{7R/8}} \bar{g}^{1-\tau}(v) \,\mathrm{d}x
\end{equation*}
for some $C = C(n, p, q, s, \Lambda, \tau) > 0$. Note that the dependence of $C$ on $\tau$ can be removed by the choice of $\tau$. Since $(u+d)_{-} \leq u_{-}$ and $u \geq 0$ in $B_{R}$, we obtain by using \eqref{eq-bar-g-comp} and recalling \eqref{eq-d} that
\begin{equation*}
T((u+d)_-; x_0, 7R/8) \leq C T(u_-; x_0, R) \leq C \bar{g}(d/R^s).
\end{equation*}
Thus, we have
\begin{equation}\label{eq-aux-I1}
I_{1} \leq C\bar{g}^{\tau}(m) \fint_{B_{7R/8}} \bar{g}^{1-\tau}(v) \,\mathrm{d}x.
\end{equation}

Let us next estimate $I_{2}$. Since
\begin{equation*}
R^{s}|D^s\eta| \leq R^{s} |\nabla \eta| |x-y|^{1-s} \leq c \frac{|x-y|^{1-s}}{R^{1-s}} \leq 2c,
\end{equation*}
we have by using \Cref{lem-G}
\begin{equation*}
I_{2} \leq C \fint_{B_{3R/4}} \int_{B_{3R/4}} g\left( \frac{1}{\varepsilon} \right) \left( \frac{R^{s}}{2c} |D^s\eta| \right)^{p} \frac{\mathrm{d}y\,\mathrm{d}x}{|x-y|^{n}}
\end{equation*}
for some $C=C(p, q, c)>0$. Note that we have
\begin{align*}
\varepsilon \geq \frac{\bar{g}^{\tau}(m)}{v(x) \bar{g}^{\tau}(v(x)) \lor v(y) \bar{g}^{\tau}(v(y))} \geq \frac{1}{v(x) \lor v(y)} \frac{\bar{g}^{\tau}(m)}{\bar{g}^{\tau}(v(x)) \lor \bar{g}^{\tau}(v(y))},
\end{align*}
and hence
\begin{equation*}
g\left( \frac{1}{\varepsilon} \right) \leq C \bar{g}(v(x) \lor v(y)) \left( \left( \frac{\bar{g}^{\tau}(v(x)) \lor \bar{g}^{\tau}(v(y))}{\bar{g}^{\tau}(m)} \right)^{p-1} \lor \left( \frac{\bar{g}^{\tau}(v(x)) \lor \bar{g}^{\tau}(v(y))}{\bar{g}^{\tau}(m)} \right)^{q-1} \right)
\end{equation*}
by \eqref{eq-bar-g-comp} and \Cref{lem-G}. This leads us to
\begin{align}\label{eq-aux-I2}
\begin{split}
I_{2}
&\leq C \fint_{B_{3R/4}} \left( \frac{\bar{g}^{1+\tau(p-1)}(v(x))}{\bar{g}^{\tau(p-1)}(m)} \lor \frac{\bar{g}^{1+\tau(q-1)}(v(x))}{\bar{g}^{\tau(q-1)}(m)} \right) \int_{B_{3R/4}} R^{sp}|D^s\eta|^p \frac{\mathrm{d}y \,\mathrm{d}x}{|x-y|^n} \\
&\leq \frac{C}{\bar{g}^{\tau(p-1)}(m)} \fint_{B_{7R/8}} \bar{g}^{1+\tau(p-1)}(v) \,\mathrm{d}x + \frac{C}{\bar{g}^{\tau(q-1)}(m)} \fint_{B_{7R/8}} \bar{g}^{1+\tau(q-1)}(v) \,\mathrm{d}x.
\end{split}
\end{align}
Combining \eqref{eq-aux-I1} and \eqref{eq-aux-I2} yields
\begin{equation*}
A \leq C \left( \frac{\fint_{B_{7R/8}} \bar{g}^{1-\tau}(v) \,\mathrm{d}x}{\bar{g}^{1-\tau}(m)} + \frac{\fint_{B_{7R/8}} \bar{g}^{1+\tau(p-1)}(v) \,\mathrm{d}x}{\bar{g}^{1+\tau(p-1)}(m)} + \frac{\fint_{B_{7R/8}} \bar{g}^{1+\tau(q-1)}(v) \,\mathrm{d}x}{\bar{g}^{1+\tau(q-1)}(m)} \right) \bar{g}(m).
\end{equation*}
We want to apply the weak Harnack inequality (\Cref{thm-WHI}) to $u+d$ with $\delta = 1-\tau$, $\delta=1+\tau(p-1)$, and $\delta = 1+\tau(q-1)$. Indeed, these values are admissible by the choice \eqref{eq-tau} of $\tau$. Therefore, \Cref{thm-WHI} with $\tau_{1}=7/8$ and $\tau_{2}=1/2$ yields
\begin{equation*}
A \leq C \left( 1 + \frac{\bar{g}^{1-\tau}(T)}{\bar{g}^{1-\tau}(m)} + \frac{\bar{g}^{1+\tau(p-1)}(T)}{\bar{g}^{1+\tau(p-1)}(m)} + \frac{\bar{g}^{1+\tau(q-1)}(T)}{\bar{g}^{1+\tau(q-1)}(m)} \right) \bar{g}(m),
\end{equation*}
where $T = R^{-s} \mathrm{Tail}_{g}((u+d)_{-}; x_{0}, R)$. Since $T \leq d/R^{s} \leq m$, we conclude
\begin{equation*}
A \leq C \bar{g}(m) \leq C g\left( \inf_{B_{R/2}} \frac{u}{R^{s}} \right) + C g\left( \frac{d}{R^{s}} \right)
\end{equation*}
by using \eqref{eq-bar-g-comp} and \Cref{lem-G}.
\end{proof}

We now take the measure data into account.

\begin{lemma}\label{lem-lower-key}
Let $u$ be a superharmonic function in $B_{R}=B_{R}(x_{0})$ such that $u \geq 0$ in $B_{R}$ and let $\mu=\mathcal{L}u$ be the Riesz measure, then
\begin{equation*}
	\frac{\mu(B_{R/2})}{R^{n-s}} \leq C g \left( \frac{\inf_{B_{R/2}}u-\inf_{B_R}u}{R^{s}} \right) + Cg\left(\frac{\mathrm{Tail}_g((u-\inf_{B_R}u)_-; x_0, R)}{R^s}\right)
\end{equation*}
for some $C=C(n, p, q, s, \Lambda) > 0$.
\end{lemma}

\begin{proof}
By considering $\tilde{u}=u-\inf_{B_R}u$, we may assume $\inf_{B_R}u=0$. Let $b=\inf_{B_{R/2}} u$ and define $v=u\land b$, then
\begin{equation}\label{eq-v-bounds}
0 \leq v \leq b \quad\text{in}~B_{R} \quad\text{and}\quad v=b \quad\text{on}~B_{R/2}.
\end{equation}
Let $\eta \in C^{\infty}_{c}(B_{5R/8})$ be a cut-off function such that $0 \leq \eta \leq 1$, $\eta \equiv 1$ on $B_{R/2}$, and $|\nabla \eta| \leq C/R$. By using $\varphi = v\eta^{q}$ as a test function, we obtain
\begin{align*}
\int_{B_{R}} \varphi \,\mathrm{d}\mu
&= \int_{B_{3R/4}} \int_{B_{3R/4}} g(|D^su|) \frac{D^su}{|D^su|} D^sv \, \eta^{q}(x) \frac{k(x, y)}{|x-y|^{n}} \,\mathrm{d}y \,\mathrm{d}x \\
&\quad + \int_{B_{3R/4}} \int_{B_{3R/4}} g(|D^su|) \frac{D^su}{|D^su|} D^s\eta^q \, v(y) \frac{k(x, y)}{|x-y|^{n}} \,\mathrm{d}y \,\mathrm{d}x \\
&\quad + 2\int_{B_{3R/4}} \int_{\mathbb{R}^{n} \setminus B_{3R/4}} g(|D^su|) \frac{D^su}{|D^su|} D^s\varphi \frac{k(x, y)}{|x-y|^{n}} \,\mathrm{d}y \,\mathrm{d}x \\
&=: I_{1} + I_{2} + I_{3}.
\end{align*}

For $I_{1}$, we observe that
\begin{equation}\label{eq-lower-I1}
g(|D^su|) \frac{D^su}{|D^su|} D^sv \leq g(|D^su|)|D^su|.
\end{equation}
Indeed, one may assume $u(x) > u(y)$ and easily check the inequality \eqref{eq-lower-I1} for the cases $u(x), u(y) > b$, $u(x) > b \geq u(y)$, and $u(x), u(y) \leq b$. By using \eqref{eq-lower-I1} and \eqref{eq-pq}, we obtain
\begin{equation*}
I_{1} \leq q\Lambda \int_{B_{3R/4}} \int_{B_{3R/4}} G(|D^sw_-|) \eta^q(x) \frac{\mathrm{d}y \,\mathrm{d}x}{|x-y|^{n}},
\end{equation*}
where $w(x) = u(x)-b$. We apply \Cref{lem-Caccioppoli} to $w$ with $l=b$ and use \Cref{lem-G}, then
\begin{align*}
I_{1}
&\leq C \int_{B_{3R/4}} \int_{B_{3R/4}} G\left( \frac{w_{-}(x) \lor w_{-}(y)}{(3R/4)^{s}} \right) ((3R/4)^s |D^s\eta|)^q \frac{\mathrm{d}y\,\mathrm{d}x}{|x-y|^{n}} \\
&\quad + C \left( \int_{B_{3R/4}} w_{-}(x) \,\mathrm{d}x \right) \sup_{x \in \mathrm{supp}\,\eta} \int_{\mathbb{R}^{n} \setminus B_{3R/4}} g\left( \frac{w_{-}(y)}{|x-y|^{s}} \right) \frac{\mathrm{d}y}{|x-y|^{n+s}} \\
&\leq C \int_{B_{R}} G\left( \frac{w_{-}}{R^{s}} \right) \,\mathrm{d}x + \frac{C}{R^{s}} \left( \int_{B_{R}} w_{-} \,\mathrm{d}x \right)T(w_-; x_0, R).
\end{align*}
Since $w_-=b-v \leq b$ in $B_{R}$ and
\begin{align*}
T(w_-; x_0, R)
&\leq T(u_-; x_0, R) + R^{s} \int_{\mathbb{R}^{n} \setminus B_{R}(x_{0})} g\left( \frac{b}{|y-x_{0}|^{s}} \right) \frac{\mathrm{d}y}{|y-x_{0}|^{n+s}} \\
&\leq T(u_-; x_0, R)+ Cg\left( \frac{b}{R^{s}} \right),
\end{align*}
we derive
\begin{equation*}
I_{1} \leq Cb R^{n-s} \left( g\left( \frac{b}{R^{s}} \right) + T(u_-; x_0, R) \right).
\end{equation*}

For $I_{2}$, we apply \Cref{lem-aux} to $u$, which is nonnegative in $B_{R}$, and use \eqref{eq-v-bounds} to have
\begin{align*}
I_{2}
&\leq C \frac{b}{R^{s}} \int_{B_{3R/4}} \int_{B_{3R/4}} g(|D^su|) R^{s} |D^s\eta| \frac{\mathrm{d}y \,\mathrm{d}x}{|x-y|^{n}} \\
&\leq Cb R^{n-s} \left( g\left( \frac{b}{R^{s}} \right) + T(u_-; x_0, R)\right).
\end{align*}
Finally, we estimate $I_3$ as follows:
\begin{align*}
	I_3 &\leq C\int_{B_{3R/4}} \int_{\mathbb{R}^{n} \setminus B_{3R/4}} g\left( \frac{u(x)}{|x-y|^{s}} \right) \frac{v(x)\eta^q(x)}{|x-y|^{s}} \frac{\mathrm{d}y \,\mathrm{d}x}{|x-y|^{n}} \\
	&\quad + C\int_{B_{3R/4}} \int_{\mathbb{R}^{n} \setminus B_{3R/4}} g\left( \frac{u_-(y)}{|x-y|^{s}} \right) \frac{v(x)\eta^q(x)}{|x-y|^{s}} \frac{\mathrm{d}y \,\mathrm{d}x}{|x-y|^{n}} \\
	&\leq CbR^{n-s}\fint_{B_{3R/4}}  g\left(\frac{u(x)}{R^s}\right) \,\mathrm{d}x+Cb R^{n-s} T(u_-; x_0, R)\\
	&\leq Cb R^{n-s} T(u_-; x_0, R),
\end{align*}
where we used \eqref{eq-v-bounds}, \eqref{eq-bar-g-comp}, the weak Harnack inequality (\Cref{thm-WHI}), and $\inf_{B_R}u=0$.

We now combine all these estimates to deduce
\begin{equation*}
b \mu(B_{R/2}) \leq \int_{B_{R}} \varphi \,\mathrm{d}\mu \leq CbR^{n-s} \left(g\left( \frac{b}{R^{s}} \right) +T(u_-; x_0, R)\right),
\end{equation*}
which leads us to
\begin{equation*}
\frac{\mu(B_{R/2})}{R^{n-s}} \leq C g \left( \frac{b}{R^{s}} \right) + CT(u_-; x_0, R),
\end{equation*}
where $C=C(n, p, q, s, \Lambda) > 0$.
\end{proof}

Let us prove \Cref{thm-Wolff-lower} by applying \Cref{lem-lower-key} iteratively.

\begin{proof}[Proof of \Cref{thm-Wolff-lower}]
Let $R_{j}=2^{-j}R$, $B^{j}=B_{R_{j}}$, and $a_{j} = \inf_{B^{j}} u$ for $j=0,1, \dots$. By applying \Cref{lem-lower-key}, we have
\begin{equation*}
	a_k \geq \sum_{j=1}^{k} (a_j-a_{j-1}) \geq c\sum_{j=1}^{k} R_j^s g^{-1}\left(\frac{\mu(B^{j})}{R_j^{n-s}}\right)-C\sum_{j=1}^{k} T_{j}
\end{equation*}
for some $c, C>0$ depending only on $n$, $p$, $q$, $s$, and $\Lambda$, where
\begin{equation*}
T_{j} = \mathrm{Tail}_g((u-a_{j-1})_-; x_0, R_{j-1}).
\end{equation*}
To estimate the tail terms, we repeat the method exploited in the proof of \Cref{thm-Wolff-upper}. Indeed, we observe that for $j=1, \cdots, k$,
\begin{align*}
g\left( \frac{T_{j}}{R_{j-1}^{s}} \right)
&= R_{j-1}^{s} \int_{\mathbb{R}^{n} \setminus B_{R}} g\left( \frac{(u(y)-a_{j-1})_{-}}{|y-x_{0}|^{s}} \right) \frac{\mathrm{d}y}{|y-x_{0}|^{n+s}} \\
&\quad + R_{j-1}^{s} \sum_{i=2}^{j} \int_{B^{j-i} \setminus B^{j-i+1}} g\left( \frac{(u(y)-a_{j-1})_{-}}{|y-x_{0}|^{s}} \right) \frac{\mathrm{d}y}{|y-x_{0}|^{n+s}} \\
&\leq C 2^{-sj} g\left( \frac{\mathrm{Tail}_g(u_-; x_0, R)}{R^{s}} \right) + C2^{-sj} g\left( \frac{a_{j-1}}{R^{s}} \right) + C \sum_{i=2}^{j} 2^{-si} g\left( \frac{a_{j-1}-a_{j-i}}{R_{j-i+1}^{s}} \right).
\end{align*}
Hence, by applying \Cref{lem-g-dyda}~(i) and then (ii) with $\beta=2^{\frac{s}{2(q-1)}}$, we have
\begin{align*}
T_{j}
&\leq C 2^{-\frac{sq}{q-1}j} (\mathrm{Tail}_{g}(u_{-}; x_{0}, R) + a_{k}) + CR_{j}^{s} g^{-1} \left( \sum_{i=2}^{j} 2^{-si} g\left( \frac{a_{j-1}-a_{j-i}}{R_{j-i+1}^{s}} \right) \right) \\
&\leq C 2^{-\frac{sq}{q-1}j} (\mathrm{Tail}_{g}(u_{-}; x_{0}, R) + a_{k}) + C \sum_{i=2}^{j} 2^{-\frac{s}{2(q-1)}i-si} (a_{j-1}-a_{j-i}).
\end{align*}
Since
\begin{equation*}
\sum_{j=1}^{k} \sum_{i=2}^{j} 2^{-\frac{s}{2(q-1)}i-si} (a_{j-1}-a_{j-i}) = \sum_{i=2}^{k} 2^{-\frac{s}{2(q-1)}i-si} \sum_{j=i}^{k} (a_{j-1}-a_{j-i}) \leq Ca_{k},
\end{equation*}
we obtain
\begin{equation*}
\sum_{j=1}^{k}T_{j} \leq C a_{k} + C\, \mathrm{Tail}_{g}(u_{-}; x_{0}, R),
\end{equation*}
which leads us to
\begin{equation*}
\sum_{j=1}^{k} R_j^s g^{-1}\left(\frac{\mu(B^{j})}{R_j^{n-s}}\right) \leq C a_{k} + C\, \mathrm{Tail}_{g}(u_{-}; x_{0}, R).
\end{equation*}
By taking the limit $k \to \infty$, we arrive at
\begin{align*}
{\bf W}^{\mu}_{s, G}(x_{0}, R/2)
&= \sum_{j=1}^{\infty} \int_{R_{j+1}}^{R_{j}} g^{-1} \left( \frac{\mu(B_{\rho}(x_0))}{\rho^{n-s}} \right) \,\frac{\mathrm{d}\rho}{\rho^{1-s}} \\
&\leq C \sum_{j=1}^{\infty} R_{j}^{s} g^{-1} \left( \frac{\mu(B^{j})}{R_{j}^{n-s}} \right) \\
&\leq C \lim_{k \to \infty} a_{k} + C \,\mathrm{Tail}_{g}(u_{-}; x_{0}, R) \\
&= Cu(x_{0}) + C \,\mathrm{Tail}_{g}(u_{-}; x_{0}, R),
\end{align*}
where $C = C(n, p, q, s, \Lambda) > 0$.
\end{proof}

\section{\texorpdfstring{$(s, G)$}{(s, G)}-capacity and \texorpdfstring{$\mathcal{L}$}{L}-potential}\label{sec-capacity}

We introduce several potential theoretical tools such as $(s, G)$-capacity and the $\mathcal{L}$-potential in our framework. They play a crucial role in the development of nonlocal nonlinear potential theory with Orlicz growth. We provide several fundamental properties of them.

\subsection{\texorpdfstring{$(s, G)$}{(s, G)}-capacity}\label{sec-sG-cap}

We begin with the definition of $(s, G)$-capacity.

\begin{definition}\label{def-capacity}
Let $\Omega \subset \mathbb{R}^n$ be an open set and let $K\subset \Omega$ be a compact set. The \emph{$(s, G)$-capacity of $K$ with respect to $\Omega$} is defined by
\begin{equation*}
\mathrm{cap}_{s, G}(K, \Omega) = \inf_{v \in W(K, \Omega)} \varrho_{W^{s, G}(\mathbb{R}^n)}(v),
\end{equation*}
where $W(K, \Omega):=\{v \in C_c^{\infty}(\Omega) : \text{$v \geq 1$ on $K$} \}$. If $U \subset \Omega$ is open, we define
	\begin{equation*}
		\mathrm{cap}_{s, G}(U, \Omega)=\sup_{K \subset U~\text{compact}}\mathrm{cap}_{s, G}(K, \Omega),
	\end{equation*}
and for an arbitrary set $E \subset \Omega$,
\begin{equation*}
		\mathrm{cap}_{s, G}(E, \Omega)=\inf_{E \subset U \subset \Omega, \ \text{$U$ open}}\mathrm{cap}_{s, G}(U, \Omega).
\end{equation*}
\end{definition}

Obviously, $\mathrm{cap}_{s, G}(\cdot, \cdot)$ is increasing in the first argument and decreasing in the second argument with respect to the set inclusion. Moreover, by approximation, the set $W(K, \Omega)$ in \Cref{def-capacity} can be replaced by
\begin{equation*}
	W_0(K, \Omega) := \overline{W(K, \Omega)}^{W^{s, G}(\mathbb{R}^n)}.
\end{equation*}
A function in $W_0(K, \Omega)$ is said to be {\it admissible for $\mathrm{cap}_{s, G}(K, \Omega)$}.

Recall that the $(s, p)$-capacity in the standard $p$-growth case inherits some properties of $p$-capacity with weight $|t|^{p-1-sp}$ by an identification of these two capacities; see Bj\"orn~\cite[Lemma~2.2]{Bjo24}, Kim--Lee--Lee~\cite[Lemma~2.17]{KLL23}, and Maz'ya~\cite[Theorem~1,~p.~512]{Maz11} for instance. However, to the best of our knowledge, such an identification is not available for the $(s, G)$-capacity in the literature. Therefore, we develop the theory by taking a more elementary and direct approach.

\begin{theorem}[Basic properties of $(s, G)$-capacity]\label{thm-capacity}
The following hold.
	\begin{enumerate}[(i)]		
		\item If $K_1$ and $K_2$ are compact subsets of $\Omega$, then 
		\begin{equation*}
			\mathrm{cap}_{s, G}(K_1 \cup K_2, \Omega)+\mathrm{cap}_{s, G}(K_1 \cap K_2, \Omega) \leq \mathrm{cap}_{s, G}(K_1, \Omega)+\mathrm{cap}_{s, G}(K_2, \Omega).
		\end{equation*}
		
		\item If $\{K_i\}_i$ is a decreasing sequence of compact subsets of $\Omega$ with $K=\cap_i K_i$, then
		\begin{equation*}
			\mathrm{cap}_{s, G}(K, \Omega)=\lim_{i \to \infty}\mathrm{cap}_{s, G}(K_i, \Omega).
		\end{equation*}
	
		\item (Subadditivity) If $E=\cup_i E_i \subset \Omega$, then
		\begin{equation*}
			\mathrm{cap}_{s, G}(E, \Omega) \leq \sum_i \mathrm{cap}_{s, G}(E_i, \Omega).
		\end{equation*}
	\end{enumerate}
\end{theorem}

\Cref{thm-capacity} says that $\mathrm{cap}_{s, G}$ is a \emph{Choquet capacity}.

\begin{proof}
(i)		Let $u_1 \in W(K_1, \Omega)$ and $u_2 \in W(K_2, \Omega)$. An application of \Cref{lem-ineq3} yields
		\begin{align*}
			&\int_{\mathbb{R}^n}\int_{\mathbb{R}^n}G(|D^s(u_1 \lor u_2)|) \,\frac{\mathrm{d}y\,\mathrm{d}x}{|x-y|^n}+\int_{\mathbb{R}^n}\int_{\mathbb{R}^n}G(|D^s(u_1 \land u_2)|) \,\frac{\mathrm{d}y\,\mathrm{d}x}{|x-y|^n} \\
			&\leq \int_{\mathbb{R}^n}\int_{\mathbb{R}^n}G(|D^su_1|) \,\frac{\mathrm{d}y\,\mathrm{d}x}{|x-y|^n}+\int_{\mathbb{R}^n}\int_{\mathbb{R}^n}G(|D^su_2|)\, \frac{\mathrm{d}y\,\mathrm{d}x}{|x-y|^n}.
		\end{align*}
	Since $u_1 \lor u_2$ and $u_1 \land u_2$ are admissible for $\mathrm{cap}_{s, G}(K_1 \cup K_2, \Omega)$ and $\mathrm{cap}_{s, G}(K_1 \cap K_2, \Omega)$, respectively, we have that
	\begin{align*}
		&\mathrm{cap}_{s, G}(K_1 \cup K_2, \Omega)+\mathrm{cap}_{s, G}(K_1 \cap K_2, \Omega) \\
		&\leq \int_{\mathbb{R}^n}\int_{\mathbb{R}^n}G(|D^su_1|) \,\frac{\mathrm{d}y\,\mathrm{d}x}{|x-y|^n}+\int_{\mathbb{R}^n}\int_{\mathbb{R}^n}G(|D^su_2|) \,\frac{\mathrm{d}y\,\mathrm{d}x}{|x-y|^n}.
	\end{align*}
		The desired inequality follows by taking the infimum over all admissible functions $u_1$ and $u_2$.
		
	Since the proofs of (ii) and (iii) are essentially independent of the nonlocal structure of operators, we refer the reader to the proofs of Theorem~2.2 and Lemma~2.3 in Heinonen--Kilpel\"ainen--Martio~\cite{HKM06}.
\end{proof}

The following lemma provides an estimate of the $(s, G)$-capacity of a closed ball $\overline{B_r(x_0)}$ with respect to a concentric ball $B_R(x_0)$.

\begin{lemma}\label{lem-cap}
	Let $R/8\leq r \leq R/2$. There exists a constant $C=C(n, p, q, s)\geq 1$ such that
	\begin{equation*}
		C^{-1} r^{n}G(r^{-s}) \leq \mathrm{cap}_{s, G}(\overline{B_{r}(x_0)}, B_{R}(x_0)) \leq C r^{n}G(r^{-s}).
	\end{equation*}
\end{lemma}

\begin{proof}
	For simplicity, we write $B_{r}=B_r(x_0)$. For the upper bound, we choose a cut-off function $\eta \in C_c^{\infty}(B_{2r})$ such that $0 \leq \eta \leq 1$, $\eta \equiv 1$ on $B_r$, and $|D \eta| \leq C/r$. Since $\eta$ is admissible for $\mathrm{cap}_{s, G}(\overline{B_r}, B_R)$, we have
	\begin{equation*}
	\mathrm{cap}_{s, G}(\overline{B_r}, B_R) \leq \int_{\mathbb{R}^n}\int_{\mathbb{R}^n}G(|D^s\eta|) \,\frac{\mathrm{d}y\,\mathrm{d}x}{|x-y|^n}.
	\end{equation*}
	Moreover, the right-hand side of the display above is estimated as
	\begin{equation*}
	\int_{\mathbb{R}^n}\int_{\mathbb{R}^n}G(|D^s\eta|) \,\frac{\mathrm{d}y\,\mathrm{d}x}{|x-y|^n} \leq C r^n G(r^{-s})
	\end{equation*}
	by using $|D^s\eta| \leq C (|x-y|/r)^{1-s} r^{-s}$, $|\eta(x)-\eta(y)|\leq 1$, and \Cref{lem-G}. This implies the upper bound.

For the lower bound, we let $u_j \in W(\overline{B_r}, B_R)$ be such that
\begin{equation*}
	\mathrm{cap}_{s, G}(\overline{B_r}, B_R)+1/j \geq \int_{\mathbb{R}^n}\int_{\mathbb{R}^n}G(|D^su_j|) \,\frac{\mathrm{d}y\,\mathrm{d}x}{|x-y|^n}
\end{equation*}
for each $j \in \mathbb{N}$. Then we obtain, for any $\sigma \in (0,s)$,
\begin{align*}
\mathrm{cap}_{s, G}(\overline{B_r}, B_R)+1/j
&\geq \int_{B_R} \int_{B_{2R}} G(|D^su_j|) \frac{|x-y|^{\sigma p}}{(3R)^{\sigma p}} \frac{\mathrm{d}y \,\mathrm{d}x}{|x-y|^n} \\
&\quad + \int_{B_R} \int_{\mathbb{R}^n \setminus B_{2R}} G(|D^su_j|) \frac{R^{\sigma p}}{|x-y|^{\sigma p}} \frac{\mathrm{d}y \,\mathrm{d}x}{|x-y|^n} \\
&= C R^n \iint_{B_R \times B_{2R}} H(|D^su_j|^p) \,\mathrm{d}\mu_1 + R^n \iint_{B_R \times B_{2R}^c} H(|D^su_j|^p) \,\mathrm{d}\mu_2,
\end{align*}
where we set
\begin{equation*}
\mathrm{d}\mu_1 = \frac{|x-y|^{\sigma p-n}}{R^{n+\sigma p}} \,\mathrm{d}x \,\mathrm{d}y, \quad \mathrm{d}\mu_2 = \frac{R^{\sigma p-n}}{|x-y|^{n+\sigma p}} \,\mathrm{d}x \,\mathrm{d}y,
\end{equation*}
and $H(t):=G(t^{1/p})$. Note that $\mu_1(B_R \times B_{2R}), \mu_2(B_R \times B_{2R}^c) < \infty$. Therefore, by applying the Jensen's inequality, we have
\begin{align*}
\mathrm{cap}_{s, G}(\overline{B_r}, B_R)+1/j
&\geq CR^n H\left( \iint_{B_R \times B_{2R}} |D^su_j|^p \,\mathrm{d}\mu_1 \right) + CR^n H\left( \iint_{B_R \times B_{2R}^c} |D^s u_j|^p \,\mathrm{d}\mu_2 \right) \\
&\geq CR^n H \left(R^{-n-\sigma p} \int_{B_R}\int_{B_{2R}}|D^{s-\sigma}u_j|^p \frac{\mathrm{d}y\,\mathrm{d}x}{|x-y|^n} \right) \\
&\quad + CR^n H \left(R^{-n+\sigma p} \int_{B_R}\int_{B_{2R}^c}|D^{s+\sigma}u_j|^p \frac{\mathrm{d}y\,\mathrm{d}x}{|x-y|^n} \right).
\end{align*}

On the other hand, since $u_j$ is admissible for $\mathrm{cap}_{s, p}(\overline{B_r}, B_R)$, we obtain
\begin{align*}
CR^{n-sp}
&\leq \mathrm{cap}_{s, p}(\overline{B_r}, B_R) \leq \int_{\mathbb{R}^n}\int_{\mathbb{R}^n}|D^su_j|^p \frac{\mathrm{d}y\,\mathrm{d}x}{|x-y|^n} \\
&\leq 2 \int_{B_R} \int_{B_{2R}} |D^s u_j|^p \frac{\mathrm{d}y\,\mathrm{d}x}{|x-y|^n} + 2 \int_{B_R} \int_{B_{2R}^c} |D^s u_j|^p \frac{\mathrm{d}y\,\mathrm{d}x}{|x-y|^n}.
\end{align*}
The first inequality follows from Lemma~2.17 in Kim--Lee--Lee~\cite{KLL23}. Hence, at least one of the following holds:
\begin{equation}\label{eq-alt1}
		\int_{B_R}\int_{B_{2R}} |D^s u_j|^p \frac{\mathrm{d}y\,\mathrm{d}x}{|x-y|^n} \geq \frac{C}{4}R^{n-sp}
\end{equation}
or
\begin{equation}\label{eq-alt2}
	\int_{B_R}\int_{B_{2R}^c} |D^su_j|^p \frac{\mathrm{d}y\,\mathrm{d}x}{|x-y|^n} \geq \frac{C}{4}R^{n-sp}.
\end{equation}
Suppose that \eqref{eq-alt1} holds. Since the map
\begin{equation*}
	\theta \mapsto R^{\theta p-n} \int_{B_R}\int_{B_{2R}} |D^\theta u_j|^p \frac{\mathrm{d}y\,\mathrm{d}x}{|x-y|^n} 
\end{equation*}
is continuous at $\theta=s$, we deduce that, for sufficiently small $\sigma>0$,
\begin{equation*}
	\int_{B_R}\int_{B_{2R}} |D^{s-\sigma}u_j|^p \frac{\mathrm{d}y\,\mathrm{d}x}{|x-y|^n} \geq \frac{C}{8}R^{n-(s-\sigma)p},
\end{equation*}
which implies that
\begin{equation*}
	\mathrm{cap}_{s, G}(\overline{B_r}, B_R)+1/j \geq CR^n H(R^{-sp})=CR^nG(R^{-s}) \geq Cr^nG(r^{-s}).
\end{equation*}
The case when \eqref{eq-alt2} holds can be considered in a similar way. Letting $j \to \infty$ finishes the proof.
\end{proof}

Sets having $(s, G)$-capacity zero play a crucial role in the development of the theory. Let us provide the precise definition.

\begin{definition}
	A set $E$ is said to be of \emph{$(s, G)$-capacity zero}, or to \emph{have $(s, G)$-capacity zero}, if 
	\begin{equation*}
		\mathrm{cap}_{s, G}(E \cap \Omega, \Omega)=0
	\end{equation*}
for all open sets $\Omega \subset \mathbb{R}^n$. In this case, we write $\mathrm{cap}_{s, G}E=0$. We write $\mathrm{cap}_{s, G}E>0$ if $E$ is not of capacity zero.
\end{definition}

The next lemma shows that one needs to test only a single bounded open set $\Omega$ containing $E$ in showing that $E$ has $(s, G)$-capacity zero.

\begin{lemma}
	Let $E$ be bounded and suppose that there exists a bounded neighborhood $\Omega$ of $E$ with $\mathrm{cap}_{s, G}(E, \Omega)=0$. Then $E$ has $(s, G)$-capacity zero.
\end{lemma}

\begin{proof}
	 Let $\Omega'$ be any open subset of $\mathbb{R}^n$. By following the argument in the proof of Lemma~2.9 in Heinonen--Kilpel\"ainen--Martio~\cite{HKM06}, we may assume that $E \cap \Omega'$ is compact. Let $u \in W(E \cap \Omega', \Omega)$ and $v \in W(E \cap \Omega', \Omega')$, then $uv \in W(E \cap \Omega', \Omega')$, and so
\begin{align*}
\mathrm{cap}_{s, G}(E \cap \Omega', \Omega')
&\leq 2\int_{\Omega'} \int_{\mathbb{R}^n} G(|D^s(uv)|) \frac{\mathrm{d}y\,\mathrm{d}x}{|x-y|^n} \\
&\leq C \int_{\Omega'} \int_{\mathbb{R}^n} (G(|u(x)| |D^sv|) + G(|v(y)||D^su|)) \frac{\mathrm{d}y\,\mathrm{d}x}{|x-y|^n} \\
&\leq C \int_{\Omega'} G(|u(x)|) \,\mathrm{d}x + C \int_{\Omega'} \int_{\mathbb{R}^n} G(|D^su|) \frac{\mathrm{d}y\,\mathrm{d}x}{|x-y|^n}.
\end{align*}
The fractional Poincar\'e inequality (\Cref{thm-Poincare}) shows that
\begin{equation*}
	\mathrm{cap}_{s, G}(E \cap \Omega', \Omega') \leq C\int_{\mathbb{R}^n}\int_{\mathbb{R}^n}G(|D^su|) \frac{\mathrm{d}y\,\mathrm{d}x}{|x-y|^n}.
\end{equation*}
Since the last integral can be chosen to be arbitrarily small, we conclude $\mathrm{cap}_{s, G}(E \cap \Omega', \Omega')=0$.
\end{proof}

In potential theory, the notion of capacity is used to refine the almost everywhere equivalence of Sobolev functions. For instance, for each $u \in W^{1, p}_0(\Omega)$ there exists a function $v \in W^{1, p}_0(\Omega)$ such that $v=u$ a.e.\ in $\Omega$ and that $v$ is continuous when restricted to a set whose complement has arbitrarily small $p$-capacity; see Theorem~4.5 in Heinonen--Kilpel\"ainen--Martio~\cite{HKM06}. We extend the theory to fractional Sobolev functions in $V^{s, G}_0(\Omega)$. For this purpose, we provide several definitions.

\begin{definition}\label{def-quasi}
\begin{enumerate}[(i)]
\item
A function $u: \Omega \to [-\infty, \infty]$ is said to be \emph{$(s, G)$-quasicontinuous in $\Omega$} if for every $\varepsilon>0$ there exists an open set $D\subset \Omega$ such that $\mathrm{cap}_{s, G}(D, \Omega) < \varepsilon$ and that $u|_{\Omega \setminus D}$ is finite valued and continuous.
\item
A sequence of functions $\psi_j: \Omega \to \mathbb{R}$ is said to converge \emph{$(s, G)$-quasiuniformly in $\Omega$} to a function $\psi$ if for every $\varepsilon>0$ there exists an open set $D \subset \Omega$ such that $\mathrm{cap}_{s, G}(D, \Omega) < \varepsilon$ and $\psi_j \to \psi$ uniformly in $\Omega \setminus D$ as $j \to \infty$.
\item
We say that a property holds \emph{$(s, G)$-quasieverywhere} or simply q.e.\ if it holds except on a set of $(s, G)$-capacity zero.
\end{enumerate}
\end{definition}

Notice that we do not introduce the so-called Sobolev capacity in \Cref{def-quasi}, but we use condenser capacity $\mathrm{cap}_{s, G}$. This is because we are only interested in $u \in V^{s, G}_0(\Omega)$ whose norm is comparable to $[u]_{V^{s, G}(\Omega)}$ by \Cref{thm-Poincare}. In order to establish a similar result for $u \in V^{s, G}(\Omega)$, one would need to introduce the Sobolev capacity.

\begin{theorem}\label{thm-quasi}
Let $\psi_j \in C_c^\infty(\Omega)$ be a Cauchy sequence in $V^{s, G}(\Omega)$. There exists a subsequence $\psi_{j_k}$ and a function $v \in V^{s, G}_0(\Omega)$, which is $(s, G)$-quasicontinuous in $\Omega$, such that $\psi_{j_k}$ converges $(s, G)$-quasiuniformly in $\Omega$ and pointwise q.e.\ in $\Omega$ to $v$ as $k \to \infty$.
\end{theorem}

\begin{proof}
	Since $\psi_j$ is a Cauchy sequence in $V^{s, G}(\Omega)$, there exists a subsequence, which we still denote by $\psi_j$, such that the series
	\begin{equation*}
		\sum_{j=1}^{\infty} \int_\Omega \int_{\mathbb{R}^n} G\left(2^j|D^s(\psi_j - \psi_{j+1})| \right) \frac{\mathrm{d}y\,\mathrm{d}x}{|x-y|^n}
	\end{equation*}
	converges. Thus, for every $\varepsilon>0$, there exists a $j_{\varepsilon} \in \mathbb{N}$ such that
	\begin{equation*}
		\sum_{j=j_\varepsilon}^{\infty} \int_\Omega \int_{\mathbb{R}^n} G\left(2^j|D^s(\psi_j - \psi_{j+1})| \right) \frac{\mathrm{d}y\,\mathrm{d}x}{|x-y|^n} < \frac{\varepsilon}{4}.
	\end{equation*}
We define $\varphi_j = 2^j(\psi_j-\psi_{j+1})$ and set
\begin{align*}
	&E_j^+:=\{ x \in \Omega: \varphi_j(x) > 1 \},\\
	&E_j^-:=\{ x \in \Omega : \varphi_j(x) < -1\}.
\end{align*}
Then by definition of $(s, G)$-capacity, we have
\begin{align*}
	&\mathrm{cap}_{s, G}(E_j^+, \Omega) \leq \varrho_{W^{s, G}(\mathbb{R}^n)}(\varphi_j) \leq 2\varrho_{V^{s, G}(\Omega)}(\varphi_j), \\
	&\mathrm{cap}_{s, G}(E_j^-, \Omega) \leq \varrho_{W^{s, G}(\mathbb{R}^n)}(-\varphi_j) \leq 2\varrho_{V^{s, G}(\Omega)}(\varphi_j).
\end{align*}
By putting
\begin{equation*}
	E_{\varepsilon}=\bigcup_{j=j_{\varepsilon}}^{\infty} \left(E_j^+\cup E_j^-\right)
\end{equation*}
and applying \Cref{thm-capacity}~(iii), we obtain that
\begin{equation*}
	\mathrm{cap}_{s, G}(E_{\varepsilon}, \Omega)\leq 4\sum_{j=j_{\varepsilon}}^{\infty} \varrho_{V^{s, G}(\Omega)}(\varphi_j) < \varepsilon.
\end{equation*}
Moreover, since
\begin{equation*}
	|\psi_j-\psi_k| \leq \sum_{l=j}^{k-1}2^{-l} \leq 2^{1-j} \quad \text{in $\Omega \setminus E_{\varepsilon}$}
\end{equation*}
for $k \geq j \geq j_\varepsilon$, $\psi_j$ converges uniformly in $\Omega \setminus E_\varepsilon$. Clearly, it also converges q.e.\ in $\Omega$ to a function $u$ which is $(s, G)$-quasicontinuous in $\Omega$.
\end{proof}

As a consequence of \Cref{thm-quasi} we obtain the following theorem, which enables us to analyze the pointwise behavior of Sobolev functions.

\begin{corollary}\label{cor-quasicont}
Let $u \in V^{s, G}_0(\Omega)$. There exists a function $v \in V^{s, G}_0(\Omega)$ such that $v$ is $(s, G)$-quasicontinuous in $\Omega$ and $u=v$ a.e.\ in $\mathbb{R}^n$.
\end{corollary}

\subsection{\texorpdfstring{$\mathcal{L}$}{L}-potential}

We define the $\mathcal{L}$-potential by using $\mathcal{L}$-harmonic functions.

\begin{definition}
	Let $\Omega \subset \mathbb{R}^n$ be a bounded open set and let $K \subset \Omega$ be a compact set. Let $\psi: \mathbb{R}^n \to [0,1]$ be such that $\psi \in C_c^{\infty}(\Omega)$ and $\psi = 1$ on $K$. The $\mathcal{L}$-harmonic function in $\Omega \setminus K$ with $u-\psi \in V^{s, G}_0(\Omega\setminus K)$ is called the \emph{$\mathcal{L}$-potential of $K$ in $\Omega$} and denoted by $\mathfrak{R}(K, \Omega)$.
\end{definition}

By the comparison principle (see Lemma~5.2 in Kim--Lee~\cite{KL23} for instance), it is immediate that the definition of $\mathfrak{R}(K, \Omega)$ is independent of the particular choice of $\psi$.

We now propose useful properties of the $\mathcal{L}$-potential. For this purpose, we introduce a generalized notion of $(s,G)$-capacity associated with the operator $\mathcal{L}$, namely,
\begin{equation*}
	\mathcal{C}(K, \Omega):=\inf_{v \in W(K, \Omega)} \mathcal{E}(v),
\end{equation*}
where 
\begin{equation*}
	\mathcal{E}(v) :=\int_{\mathbb{R}^n}\int_{\mathbb{R}^n} G(|D^sv|) \,\frac{k(x, y)}{|x-y|^n}\,\mathrm{d}y\,\mathrm{d}x.
\end{equation*}
Note that $\mathcal{C}(K, \Omega)=\mathrm{cap}_{s,G}(K, \Omega)$ when $k(x,y) \equiv1$. By modifying Lemma~2.16 in Kim--Lee--Lee~\cite{KLL23}, one can easily obtain the following lemma.

\begin{lemma}\label{lem-potential}
	Let $\Omega \subset \mathbb{R}^n$ be a bounded open set and let $K \subset \Omega$ be a compact set.
	\begin{enumerate}[(i)]
		\item $\mathfrak{R}(K, \Omega)$ is admissible for $\mathrm{cap}_{s,G}(K, \Omega)$. In particular, $\mathfrak{R}(K, \Omega) \in W_0^{s, G}(\Omega)$.
		
		\item Let $\widetilde{W}(K, \Omega) := \lbrace v \in C_c^{\infty}(\Omega): v = 1 \text{ in a neighborhood of } K \rbrace$, then
		\begin{equation*}
			\mathcal{C}(K, \Omega)=\inf_{v \in \widetilde{W}(K, \Omega)} \mathcal{E}(v).
		\end{equation*}
		
		\item $\mathfrak{R}(K, \Omega)$ is a minimizer of $\mathcal{E}$ among all functions belonging to $W_0(K, \Omega)$, i.e.,
		\begin{equation*}
			\mathcal{E}(\mathfrak{R}(K, \Omega))=\mathcal{C}(K, \Omega).
		\end{equation*}
		
		\item $\mathfrak{R}(K, \Omega)$ is a supersolution of $\mathcal{L}u=0$ in $\Omega$.
	\end{enumerate}
\end{lemma}

We close the section with the following lemma, which investigates the pointwise behavior of the $\mathcal{L}$-potentials.

\begin{lemma}\label{lem-quasirepresentative}
Let $\Omega \subset \mathbb{R}^n$ be a bounded open set and let $K \subset \Omega$ be a compact set. Then $\mathfrak{R}(K, \Omega)$ has a representative $v$, which is $(s, G)$-quasicontinuous in $\Omega$. Moreover, $v \geq 1$ q.e.\ on $K$.
\end{lemma}

\begin{proof}
It follows from \Cref{cor-quasicont} and \Cref{lem-potential}.
\end{proof}

\section{Wiener criterion}\label{sec-wiener}

In this section, we establish the Wiener criterion, \Cref{thm-Wiener}, which provides a necessary and sufficient condition for a boundary point to be regular. We mainly follow the argument presented in Kim--Lee--Lee~\cite{KLL23} and use regularity estimates obtained in \Cref{sec-preliminaries} and \Cref{sec-upper} to prove \Cref{thm-Wiener}.

\subsection{The sufficient condition}

In this section, the proof of the sufficiency part of \Cref{thm-Wiener} is provided, namely, we prove that a boundary point $x_0 \in \partial \Omega$ is regular if the Wiener integral diverges. Here, the Wiener integral refers to the integral on the left-hand side of \eqref{eq-Wiener}.

Throughout the section, we fix a boundary point $x_0 \in \partial \Omega$ and assume that the Wiener integral diverges. Let $\vartheta \in V^{s, G}(\Omega) \cap C(\mathbb{R}^n)$ and let $u$ be the unique harmonic function in $\Omega$ such that $u-\vartheta \in V^{s, G}_0(\Omega)$. Since $u=\vartheta$ a.e.\ outside $\Omega$, we may assume that $u \equiv \vartheta$ outside $\Omega$. We will prove that the equality \eqref{eq-regular} holds, or equivalently,
\begin{equation}\label{eq-claim-regular}
	\lim_{\rho \to 0}\sup_{\Omega \cap B_{\rho}(x_0)}u \leq \vartheta(x_0) \quad \text{and} \quad 	\lim_{\rho \to 0} \inf_{\Omega \cap B_{\rho}(x_0)}u \geq \vartheta(x_0).
\end{equation}
It is enough to prove the first inequality of \eqref{eq-claim-regular} by symmetry. We argue by contradiction; let us assume
\begin{equation*}
	L:=\lim_{\rho \to 0}\sup_{\Omega \cap B_{\rho}(x_0)}u>\vartheta(x_0).
\end{equation*} 
Note that $L$ is finite by \Cref{thm-loc-bdd}. We choose any $l \in \mathbb{R}$ satisfying $\vartheta(x_0) <l <L$, then the continuity of $\vartheta$ shows that there exists a sufficiently small radius $r_{\ast}>0$ such that $l \geq \sup_{D_r(x_0)} \vartheta$ for any $r \in (0, r_{\ast})$, where
\begin{equation}\label{eq-Dr}
D_r(x_0) := \overline{B_r(x_0)} \setminus \Omega.
\end{equation}
Let $M_l(r):=\sup_{B_r(x_0)}(u-l)_+$, then it follows from $\lim_{r \to 0}M_l(r)=L-l>0$ that $M_l(r) \geq L-l>0$ for any $r>0$. We now define a function
\begin{equation*}
u_{l, r}:=M_l(r)-(u-l)_+,
\end{equation*}
which is a supersolution of $\mathcal{L}u_{l, r}=0$ in $\Omega$. Note that $(u_{l, r})_m^-=u_{l, r}$ since $(u-l)_+=0$ in $B_r(x_0) \setminus \Omega$. We begin with the following lemma, which is obtained by using a Caccioppoli-type estimate and the weak Harnack inequality.

\begin{lemma}\label{lem-grad-est}
There exists a constant $C = C(n, p, q, s, \Lambda, L-l) > 0$ such that
\begin{equation}\label{eq-grad-est}
	\begin{split}
		&\frac{1}{M_l(4\rho)} \int_{B_{3\rho/2}(x_0)} \int_{B_{3\rho/2}(x_0)} G(|D^su_{l, 4\rho}|) \,\frac{\mathrm{d}y \,\mathrm{d}x}{|x-y|^n}\\
		&\leq C \rho^{n-s} \left( g\left(\frac{M_{l}(4\rho)-M_l(\rho)}{\rho^s}\right)+g\left(\frac{\mathrm{Tail}_g((u_{l, 4\rho})_-; x_0, 4\rho )}{\rho^s}\right) \right)
	\end{split}
\end{equation}
for any $\rho \in (0, r_{\ast}/4)$.
\end{lemma}

\begin{proof}
We write $B_r= B_r(x_0)$ for simplicity. Let $\eta \in C_c^{\infty}(B_{13\rho/8})$ be a cut-off function satisfying $\eta \equiv 1$ on $\overline{B_{3\rho/2}}$, $0 \leq \eta \leq 1$, and $|\nabla \eta| \leq C/\rho$. Since $u_{l, 4\rho}$ is a supersolution in $\Omega$, we obtain by using $\varphi = (u-l)_+ \eta^q \in V^{s, G}_0(\Omega)$ as a test function
\begin{equation*}
	\begin{split}
		0
		&\leq \int_{B_{7\rho/4}}\int_{B_{7\rho/4}} g(|D^su_{l, 4\rho}|)\frac{D^su_{l, 4\rho}}{|D^su_{l, 4\rho}|} D^s\varphi  \frac{k(x,y)}{|x-y|^n}\,\mathrm{d}y\,\mathrm{d}x\\
		&\quad + 2\int_{B_{7\rho/4}}\int_{\mathbb{R}^n \setminus B_{7\rho/4}} g(|D^su_{l, 4\rho}|)\frac{D^su_{l, 4\rho}}{|D^su_{l, 4\rho}|} \frac{\varphi(x)}{|x-y|^s} \frac{k(x,y)}{|x-y|^n}\,\mathrm{d}y\,\mathrm{d}x.
	\end{split}
\end{equation*}
The equality $\varphi(x) - \varphi(y) = (u(x)-l)_+ (\eta^q(x) - \eta^q(y)) - (u_{l, 4\rho}(x) - u_{l, 4\rho}(y)) \eta^q(y)$ yields
\begin{equation*}
	\begin{split}
		&\frac{1}{M_l(4\rho)} \int_{B_{3\rho/2}} \int_{B_{3\rho/2}} G(|D^su_{l, 4\rho}|) \,\frac{\mathrm{d}y \,\mathrm{d}x}{|x-y|^n} \\
		&\leq C \int_{B_{7\rho/4}} \int_{B_{7\rho/4}} g(|D^su_{l, 4\rho}|)
		\frac{(u(x)-l)_+}{M_l(4\rho)} |D^s\eta^q| \frac{k(x,y)}{|x-y|^n}\,\mathrm{d}y \,\mathrm{d}x \\
		&\quad + C \int_{B_{7\rho/4}}\int_{\mathbb{R}^n \setminus B_{7\rho/4}} g(|D^su_{l, 4\rho}|)\frac{D^su_{l, 4\rho}}{|D^su_{l, 4\rho}|}
		\frac{\varphi(x)}{M_l(4\rho)|x-y|^s} \frac{k(x,y)}{|x-y|^n}\,\mathrm{d}y \,\mathrm{d}x\\
		&=: I_1 + I_2.
	\end{split}
\end{equation*}
Thus, it is enough to estimate $I_1+I_2$ by the right-hand side of \eqref{eq-grad-est}.

Let us first estimate $I_1$. Since $(u(x)-l)_+ \leq M_l(4\rho)$ and $|\eta^q(x)-\eta^q(y)| \leq C\rho^{-1}|x-y|$, we estimate $I_1$ as
\begin{equation*}
	I_1 \leq \frac{C}{\rho^s}\int_{B_{7\rho/4}} \int_{B_{7\rho/4}} g(|D^su_{l, 4\rho}|) \left( \frac{|x-y|}{\rho}\right)^{1-s}
	 \frac{\mathrm{d}y \,\mathrm{d}x}{|x-y|^{n}}.
\end{equation*}
We define $v=(u_{l, 4\rho} + d)/\rho^s$, with $d>0$ to be determined later, and let $m=\inf_{B_\rho} v$. By applying \eqref{eq-alg} with $a=|D^su_{l, 4\rho}|$, $b=(|x-y|/\rho)^{1-s}$, and
\begin{equation*}
\varepsilon=\bar{g}^{-\beta}(m) \fint_{v(y)}^{v(x)} \frac{\bar{g}^{\beta}(t)}{t} \,\mathrm{d}t,
\end{equation*}
we obtain for any $\beta \in (-1, 0)$
\begin{equation*}
g(|D^su_{l, 4\rho}|) \left( \frac{|x-y|}{\rho}\right)^{1-s} \leq q \bar{g}^{-\beta}(m) G(|D^su_{l, 4\rho}|) \fint_{v(y)}^{v(x)} \frac{\bar{g}^{\beta}(t)}{t} \,\mathrm{d}t + C g(\varepsilon^{-1}) \left( \frac{|x-y|}{\rho} \right)^{(1-s)p},
\end{equation*}
where $C=C(p, q)$ is a positive constant. Since
\begin{equation*}
	\fint_{v(y)}^{v(x)} \frac{\bar{g}^{\beta}(t)}{t} \,\mathrm{d}t  \geq \frac{\bar{g}^{\beta}(v(x))}{v(x)} \land \frac{\bar{g}^{\beta}(v(y))}{v(y)} \geq \left(\frac{v(x)}{\bar{g}^{\beta}(v(x))}+\frac{v(y)}{\bar{g}^{\beta}(v(y))}\right)^{-1}
\end{equation*}
we have
\begin{align*}
g(\varepsilon^{-1})
&\leq C \bar{g} \left( \frac{\bar{g}^{-\beta}(v(x))}{\bar{g}^{-\beta}(m)} v(x) \right) + C \bar{g}\left( \frac{\bar{g}^{-\beta}(v(y))}{\bar{g}^{-\beta}(m)} v(y) \right) \\
&\leq C \bar{g}^{\beta(q-1)}(m) \left( \bar{g}^{1-\beta(q-1)}(v(x)) + \bar{g}^{1-\beta(q-1)}(v(y))\right).
\end{align*}
Consequently,
\begin{align*}
I_1
&\leq \frac{C}{\rho^s} \bar{g}^{-\beta}(m) \int_{B_{7\rho/4}} \int_{B_{7\rho/4}} G(|D^su_{l, 4\rho}|) \fint_{v(y)}^{v(x)} \frac{\bar{g}^{\beta}(t)}{t} \,\mathrm{d}t \,\frac{\mathrm{d}y \,\mathrm{d}x}{|x-y|^{n}} \\
&\quad + \frac{C}{\rho^s}\bar{g}^{\beta(q-1)}(m) \int_{B_{7\rho/4}} \int_{B_{7\rho/4}} \bar{g}^{1-\beta(q-1)}(v(x)) \left( \frac{|x-y|}{\rho} \right)^{(1-s)p} \frac{\mathrm{d}y \,\mathrm{d}x}{|x-y|^{n}}.
\end{align*}
Assume that we have taken $d=\mathrm{Tail}_g((u_{l, 4\rho})_-; x_0, 4\rho)$, then
\begin{equation*}
g(\mathrm{Tail}_g((u_{l, 4\rho}+d)_{-}; x_{0}, 2\rho)/\rho^{s}) \leq Cg(d/\rho^s) \leq Cg(m).
\end{equation*}
Thus, by applying Lemma~3.6 in Kim--Lee~\cite{KL23} we obtain
\begin{equation*}
I_1 \leq \frac{C}{\rho^s} \bar{g}^{-\beta}(m) \int_{B_{2\rho}} \bar{g}^{1+\beta}(v) \,\mathrm{d}x + \frac{C}{\rho^s}\bar{g}^{\beta(q-1)}(m) \int_{B_{7\rho/4}} \bar{g}^{1-\beta(q-1)}(v) \,\mathrm{d}x.
\end{equation*}
If we choose $\beta \in (-1, 0)$ and assume in addition that $\beta$ is sufficiently close to 0 so that $1-\beta(q-1) \in (0, n/(n-sp))$ when $sp<n$, then the weak Harnack inequality (\Cref{thm-WHI}) can be employed to conclude
\begin{equation*}
I_1 \leq C \rho^{n-s} g(m) \leq C \rho^{n-s} \left( g\left(\frac{M_{l}(4\rho)-M_l(\rho)}{\rho^s}\right)+g\left(\frac{\mathrm{Tail}_g((u_{l, 4\rho})_-; x_0, 4\rho )}{\rho^s}\right) \right),
\end{equation*}
where $C=C(n, p, q, s, \Lambda)$ is a positive constant.

Let us next estimate $I_2$. Since
\begin{equation*}
g(|D^su_{l, 4\rho}|)\frac{D^su_{l, 4\rho}}{|D^su_{l, 4\rho}|} \leq C g\left( \frac{u_{l, 4\rho}(x)}{|x-y|^s} \right) + C g\left( \frac{(u_{l, 4\rho})_-(y)}{|x-y|^s} \right)
\end{equation*}
and $\varphi(x) \leq M_l(4\rho) \eta^q(x)$, we have
\begin{align*}
I_2
&\leq C \int_{B_{13\rho/8}} \int_{B_{7\rho/4}^c} g\left( \frac{u_{l, 4\rho}(x)}{\rho^s} \right) \frac{\mathrm{d}y\,\mathrm{d}x}{|x-y|^{n+s}} + C \int_{B_{13\rho/8}} \int_{B_{7\rho/4}^c} g\left(  \frac{(u_{l, 4\rho})_-(y)}{|x-y|^s} \right) \frac{\mathrm{d}y \,\mathrm{d}x}{|x-y|^{n+s}} \\
&\leq C \int_{B_{13\rho/8}} g\left( \frac{u_{l, 4\rho}}{\rho^s} \right) \mathrm{d}x + C \rho^{n-s} g\left( \frac{\mathrm{Tail}_g((u_{l, 4\rho})_-; x_0, 4\rho)}{\rho^s} \right).
\end{align*}
By using \Cref{thm-WHI}, we arrive at
\begin{equation*}
I_2 \leq C \rho^{n-s} \left( g\left(\frac{M_{l}(4\rho)-M_l(\rho)}{\rho^s}\right)+g\left(\frac{\mathrm{Tail}_g((u_{l, 4\rho})_-; x_0, 4\rho )}{\rho^s}\right) \right).
\end{equation*}
The desired result \eqref{eq-grad-est} now follows from the estimates for $I_1$ and $I_2$.
\end{proof}

The following is a key lemma for the proof of the sufficiency of \Cref{thm-Wiener}.

\begin{lemma}\label{lem-key}
	There exists a constant $C=C(n, p, q, s, \Lambda, L-l)>0$ such that
	\begin{equation*}
		\mathrm{cap}_{s, G}(D_{\rho}(x_0), B_{2\rho}(x_0)) \leq 
		C \rho^{n-s} \left( g\left(\frac{M_{l}(4\rho)-M_l(\rho)}{\rho^s}\right)+g\left(\frac{\mathrm{Tail}_g((u_{l, 4\rho})_-; x_0, 4\rho )}{\rho^s}\right) \right)
	\end{equation*}
for any $\rho \in (0, r_{\ast}/4)$, where $D_\rho(x_0)$ is given by \eqref{eq-Dr}.
\end{lemma}

\begin{proof}
Let us write $B_{\rho}=B_{\rho}(x_0)$ and $D_{\rho}=D_{\rho}(x_0)$ for simplicity. Let $\eta \in C_c^{\infty}(B_{3\rho/2})$ be a cut-off function such that $\eta \equiv 1$ on $\overline{B_{\rho}}$, $0 \leq \eta \leq 1$, and $|\nabla \eta| \leq C/\rho$. Since $v_{l, 4\rho}:= \eta u_{l, 4\rho}/M_l(4\rho)$ is admissible for $\mathrm{cap}_{s, G}(D_\rho, B_{2\rho})$, we have
\begin{equation*}
\mathrm{cap}_{s, G}(D_{\rho}, B_{2\rho}) \leq \varrho_{W^{s, G}(\mathbb{R}^n)}(v_{l, 4\rho}) \leq 2 \int_{B_{3\rho/2}} \int_{\mathbb{R}^n} G(|D^s v_{l, 4\rho}|) \,\frac{\mathrm{d}y \,\mathrm{d}x}{|x-y|^n} =:I.
\end{equation*}
The triangle inequality and \Cref{lem-G} yield
\begin{equation*}
I \leq C \int_{B_{3\rho/2}} \int_{\mathbb{R}^n} G\left( |D^s\eta| \frac{u_{l, 4\rho}(x)}{M_l(4\rho)} \right) \frac{\mathrm{d}y \,\mathrm{d}x}{|x-y|^n} + C \int_{B_{3\rho/2}} \int_{\mathbb{R}^n} G\left( |D^su_{l, 4\rho}| \frac{\eta(y)}{M_l(4\rho)} \right) \frac{\mathrm{d}y \,\mathrm{d}x}{|x-y|^n}.
\end{equation*}
By using $u_{l, 4\rho} \leq M_l(4\rho)$ and \Cref{lem-G} again, we obtain
\begin{equation*}
	\begin{split}
	&\int_{\mathbb{R}^n} G\left(|D^s\eta| \frac{u_{l, 4\rho}(x)}{M_{l}(4\rho)}\right) \frac{\mathrm{d}y}{|x-y|^n}\\
	& \leq \int_{B_{\rho}(x)} G\left(\frac{|x-y|^{1-s}}{\rho} \frac{u_{l, 4\rho}(x)}{M_{l}(4\rho)}\right) \frac{\mathrm{d}y}{|x-y|^n}+\int_{B_{\rho}^c(x)} G\left(\frac{u_{l, 4\rho}(x)}{|x-y|^s M_{l}(4\rho)}\right) \frac{\mathrm{d}y}{|x-y|^n}\\
	&\leq C\int_{B_{\rho}(x)} G\left(\frac{u_{l, 4\rho}(x)}{\rho^sM_{l}(4\rho)}\right) \frac{\rho^{ps-p}}{|x-y|^{n+ps-p}}\, \mathrm{d}y +C\int_{B_{\rho}^c(x)} g\left(\frac{u_{l, 4\rho}(x)}{\rho^s M_{l}(4\rho)}\right)\frac{\mathrm{d}y}{|x-y|^{n+s}}\\
	&\leq C\rho^{-s}g\left(\frac{u_{l, 4\rho}(x)}{\rho^s M_l(4\rho)}\right).
	\end{split}
\end{equation*}
Thus, it follows from $M_l(4\rho) \geq L-l$ that
\begin{equation*}
I \leq C \rho^{n-s} \fint_{B_{3\rho/2}} g\left(\frac{u_{l, 4\rho}}{\rho^s}\right) \mathrm{d}x + \frac{C}{M_l(4\rho)} \int_{B_{3\rho/2}} \int_{B_{3\rho/2}} G(|D^su_{l, 4\rho}|)\, \frac{\mathrm{d}y \,\mathrm{d}x}{|x-y|^n},
\end{equation*}
where $C=C(n, p, q, s, L-l)$ is a positive constant. We now apply \Cref{thm-WHI} with $\delta=1$ and \Cref{lem-grad-est} to conclude the lemma.
\end{proof}

We are now ready to prove the sufficiency of the Wiener criterion.

\begin{proof}[Proof of the sufficiency of \Cref{thm-Wiener}]
As explained at the beginning of this section, we assume to the contrary that \eqref{eq-claim-regular} holds and find a contradiction. By \Cref{lem-key}, we have
\begin{equation*}
\frac{\mathrm{cap}_{s, G}(D_{\rho}(x_0), B_{2\rho}(x_0))}{\rho^{n-s}} \leq C g\left(\frac{M_{l}(4\rho)-M_l(\rho)}{\rho^s}\right) + C g\left(\frac{\mathrm{Tail}_g((u_{l, 4\rho})_-; x_0, 4\rho )}{\rho^s} \right)
\end{equation*}
for any $\rho \in (0, r_{\ast}/4)$. By using this inequality, we derive an estimate for the Wiener integral:
\begin{equation} \label{eq-Wiener-I12}
	\begin{split}
		&\int_0^{r_\ast/4} \rho^s g^{-1}\left(\frac{\mathrm{cap}_{s, G}(\overline{B_{\rho}(x_0)} \setminus \Omega, B_{2\rho}(x_0)) }{\rho^{n-s}}\right) \frac{\mathrm{d}\rho}{\rho} \\
		&\leq C \int_0^{r_{\ast}/4}\frac{M_l(4\rho)-M_l(\rho)}{\rho}\,\mathrm{d}\rho + C \int_0^{r_{\ast}/4}\frac{\mathrm{Tail}_g((u_{l, 4\rho})_-; x_0, 4\rho)}{\rho}\,\mathrm{d}\rho.
	\end{split}
\end{equation}
On the one hand, one can show that the right-hand side of \eqref{eq-Wiener-I12} is finite by modifying the argument presented in Section~4 in Kim--Lee--Lee~\cite{KLL23} and using \Cref{lem-g-dyda} and the local boundedness (\Cref{thm-loc-bdd}) as in the previous sections. On the other hand, the left-hand side of \eqref{eq-Wiener-I12} diverges by \eqref{eq-Wiener}, which leads us to a contradiction.
\end{proof}

\subsection{The necessary condition}

In this section, we prove the necessity of the Wiener criterion. We basically follow the lines of proofs presented in \cite{KLL23}, where the standard growth case $G(t)=t^p$ was covered by dividing the cases into two cases $p \leq n/s$ and $p > n/s$. However, we encounter a difficulty because such a distinction is not available in the present framework. Moreover, as we pointed out in \Cref{rmk-wiener}~(iii), the case $p>n/s$ in \cite{KLL23} was treated incorrectly. We overcome such a challenge by considering two cases $\mathrm{cap}_{s, G}(\{x\})>0$ and $\mathrm{cap}_{s, G}(\{x\})=0$, and fill the gap by exploiting \Cref{lem-quasirepresentative} regarding $(s, G)$-quasicontinuity.

We first provide a necessary condition for a boundary to be regular in terms of the $\mathcal{L}$-potential. Before we state the result, we consider a lsc-regularization of the $\mathcal{L}$-potential; define $\widehat{\mathfrak{R}}(K, \Omega)$ in $\Omega$ by
\begin{equation*}
\widehat{\mathfrak{R}}(K, \Omega)(x)= 
\begin{cases}
\mathfrak{R}(K, \Omega)(x) &\text{in}~ \mathbb{R}^n \setminus \Omega, \\
\displaystyle\essliminf_{y \to x} \mathfrak{R}(K, \Omega)(y) &\text{in}~ \Omega.
\end{cases}
\end{equation*}
Note that $\widehat{\mathfrak{R}}(K, \Omega)$ in $\Omega$ is superharmonic in $\Omega$ and $\widehat{\mathfrak{R}}(K, \Omega)=\mathfrak{R}(K, \Omega)$ a.e.\ in $\Omega$ by \Cref{thm-lsc} and \Cref{thm-relation}~(i).

\begin{lemma}\label{lem-irr-zero}
	Suppose that the $(s, G)$-capacity of a single point is zero. Let $x_0 \in \partial \Omega$. If there exists $\rho>0$ such that
	\begin{equation}\label{eq-irr}
	\widehat{\mathfrak{R}}(D_\rho(x_0), B_{8\rho}(x_0))(x_0) < 1,
	\end{equation}
	then $x_0$ is irregular.
\end{lemma}

We will see that \eqref{eq-irr} does not hold for any $\rho>0$ when the $(s, G)$-capacity of a single point is positive. This will improve \Cref{lem-irr-zero}; see \Cref{lem-irr}.

\begin{proof}
For simplicity, we write $B_\rho=B_\rho(x_0)$, $D_\rho=D_\rho(x_0)$, and $u_\rho=\widehat{\mathfrak{R}}(D_\rho, B_{8\rho})$. If \eqref{eq-irr} holds for some $\rho > 0$, then it holds for $\rho' < \rho$ as well by the comparison principle; see Lemma~5.2 in Kim--Lee~\cite{KL23} for instance. Thus, we may choose $\rho_0>0$ sufficiently small so that $\Omega\cap \partial B_{8\rho} \neq \emptyset$ for all $\rho < \rho_0$. Let us fix $\rho < \rho_0$ and $r < \rho$. Let $\vartheta \in C^{\infty}_c(\mathbb{R}^n)$ be a function such that $\vartheta\equiv 3/2$ on $\overline{B_{r/2}} \setminus \Omega$, $0 \leq \vartheta<3/2$ in $(B_r \setminus \overline{B_{r/2}})\setminus \Omega$, and $\vartheta \equiv 0$ on the remaining part of $\mathbb{R}^n \setminus \Omega$. Let $v_r$ be the harmonic function in $\Omega$ such that $v_r-\vartheta \in V_0^{s,G}(\Omega)$. We claim that there exists $r_0 = r_0(n, p, q, s, \Lambda, \rho, \mathrm{diam}(\Omega)) > 0$ such that
\begin{equation}\label{eq-comparison}
v_r \leq \frac{1}{2}+u_{\rho}
\end{equation}
in $\Omega\cap B_{8\rho}$ in the weak sense for all $r < r_0$. Indeed, by following the argument given in the proof of Lemma~5.5 in Kim--Lee--Lee~\cite{KLL23}, one can easily check that \eqref{eq-comparison} holds in $\mathbb{R}^n \setminus \Omega$. Thus, it is enough to show \eqref{eq-comparison} in $\Omega \setminus B_{8\rho}$. Note that $u_\rho=0$ outside $B_{8\rho}$. Let us fix a point $z \in \Omega \setminus B_{8\rho}$ and consider a ball $B_{\rho}(z)$. Since $B_r(x_0) \cap B_{\rho}(z) = \emptyset$ and $\vartheta=0$ in $(\mathbb{R}^n \setminus B_r(x_0)) \setminus \Omega$, we have $\vartheta=0$ in $B_{\rho}(z) \setminus \Omega$. By applying the local boundedness (\Cref{thm-loc-bdd}) to $v_r$ in $B_{\rho}(z)$ and using H\"older's inequality, we obtain
	\begin{equation}\label{lb}
	\esssup_{B_{\rho/2}(z)} v_r \leq \varepsilon\, \mathrm{Tail}_g(v_r; z, \rho/2) + C(\varepsilon) \rho^s G^{-1}\left( \fint_{B_{\rho}(z)}G\left(\frac{v_r}{\rho^s}\right) \,\mathrm{d}x \right),
	\end{equation}
for any $\varepsilon>0$. 
where $\gamma=2(q-1)n/s_0>0$. In order to estimate the terms on the right-hand side of \eqref{lb}, we introduce the $\mathcal{L}$-potential $w$ of $D_r$ in $B_R$, where $R$ is the diameter of $\Omega$. By the comparison principle, $v_r \leq \frac{3}{2}w$ a.e.\ in $\mathbb{R}^n$. Since $0 \leq w \leq 1$, we have
	\begin{equation} \label{eq-vr-tail}
		\mathrm{Tail}_g(v_r; z, \rho/2) \leq  C\mathrm{Tail}_g(w; z, \rho/2) \leq C.
	\end{equation}
Moreover, by \Cref{thm-Poincare} we obtain
\begin{equation}\label{eq-poincare}
\int_{B_{\rho}(z)} G\left(\frac{v_r}{\rho^s}\right) \,\mathrm{d}x \leq C \int_{B_R(x_0)} G\left(\frac{w}{\rho^s}\right) \,\mathrm{d}x \leq C\left(\frac{R}{\rho} \right)^{sq} \varrho_{W^{s, G}(\mathbb{R}^n)}(w).
\end{equation}
Furthermore, we observe that
	\begin{equation} \label{eq-capacity}
			\varrho_{W^{s, G}(\mathbb{R}^n)}(w)
			=\mathrm{cap}_{s, G}(D_r, B_R)\leq \mathrm{cap}_{s, G}(\overline{B_r}, B_R)=:\varphi(r),
	\end{equation}
where $\lim_{r\to 0}\varphi(r)=0$ due to the assumption that a single point is of $(s, G)$-capacity zero. By combining \eqref{lb}, \eqref{eq-vr-tail}, \eqref{eq-poincare} and \eqref{eq-capacity}, we arrive at
	\begin{equation*}
		\esssup_{B_{\rho/2}(z)}v_r \leq 
	C_1 \varepsilon + C_2(\varepsilon) \rho^s G^{-1} \left( \frac{1}{\rho^n} \left( \frac{R}{\rho} \right)^{sq} \varphi(r) \right).
	\end{equation*}
	We take $\varepsilon = 1/(4C_1)$ and then take $r_0 = r_0(n, p, q, s, \Lambda, \rho, R)>0$ sufficiently small so that $C_2(\varepsilon) \rho^s G^{-1} (\rho^{-n-sq}R^{sq}\varphi(r)) \leq 1/2$. Then we conclude $\esssup_{B_{\rho/2}(z)}v_r \leq 1/2$ for all $r < r_0$, from which \eqref{eq-comparison} follows.

The comparison principle (see Lemma~5.2 in Kim--Lee~\cite{KL23} for instance) now shows that
	\begin{equation*}
		\liminf_{\Omega \ni x \to x_0} v_r <\frac{3}{2}=\vartheta(x_0),
	\end{equation*}
	which concludes that the boundary point $x_0$ is irregular.
\end{proof}

The following lemma is an improvement of \Cref{lem-irr-zero}, namely, we drop the assumption that $(s, G)$-capacity of a single point is zero.

\begin{lemma}\label{lem-irr}
If \eqref{eq-irr} holds for some $\rho>0$, then $x_0$ is irregular.
\end{lemma}

\begin{proof}
By \Cref{lem-irr-zero}, it is enough to prove that \eqref{eq-irr} does not hold for any $\rho>0$ when the $(s, G)$-capacity of a single point is positive. Assume that every single point has a positive $(s, G)$-capacity. By appealing \Cref{lem-quasirepresentative}, we may choose a representative $v_\rho$ of $\widehat{\mathfrak{R}}(D_\rho(x_0), B_{8\rho}(x_0))$ which is $(s, G)$-quasicontinuous in $B_{8\rho}(x_0)$ and not less than 1 q.e.\ on $D_\rho(x_0)$. Since a single point is of positive capacity, we deduce that $v_\rho$ is continuous in $B_{8\rho}(x_0)$ and $v_\rho \geq 1$ on $D_\rho(x_0)$. This in particular implies that $v_\rho(x_0) \geq 1$. Thus, we have
\begin{equation*}
\widehat{\mathfrak{R}}(D_\rho(x_0), B_{8\rho}(x_0))(x_0) = \liminf_{x \to x_0} \widehat{\mathfrak{R}}(D_\rho(x_0), B_{8\rho}(x_0))(x) = \liminf_{x \to x_0} v_\rho(x) = v_\rho(x_0) \geq 1,
\end{equation*}
implying that \eqref{eq-irr} does not hold.
\end{proof}

With the help of \Cref{lem-irr}, we finally prove the necessary part of \Cref{thm-Wiener}.

\begin{proof} [Proof of the necessity of \Cref{thm-Wiener}]
We prove that a boundary point $x_0 \in \partial \Omega$ is irregular if the Wiener integral converges, that is,
\begin{equation}\label{eq-wieint}
\int_0^1 t^s g^{-1}\left(\frac{\mathrm{cap}_{s, G}(D_t, B_{2t}) }{t^{n-s}}\right) \frac{\mathrm{d}t}{t} < \infty,
\end{equation}
where we denote $B_t=B_t(x_0)$ and $D_t=D_t(x_0)$. Recall that $D_t$ is defined as in \eqref{eq-Dr}. Let $u_\rho=\widehat{\mathfrak{R}}(D_\rho, B_{8\rho})$ and let $\mu_\rho=\mathcal{L}u_\rho$ be the Riesz measure in $B_{8\rho}$, then an application of \Cref{thm-Wolff-upper} to $u_\rho$ in $B_{4\rho}$ yields
	\begin{equation} \label{pt}
		u_{\rho}(x_0) \leq C \left( \inf_{B_{2\rho}}u_{\rho} + {\bf W}_{s,G}^{\mu_{\rho}}(x_0, 4\rho) + \mathrm{Tail}_g(u_{\rho}; x_0, 2\rho) \right).
	\end{equation}
By \Cref{lem-irr}, it is enough to show that the right-hand side of \eqref{pt} is less than one for some small $\rho>0$.

For the first term on the right-hand side of \eqref{pt}, we claim that
	\begin{equation}\label{eq-necessary-cliam}
		\liminf_{\rho \to 0} \inf_{B_{2\rho}}u_{\rho}=0.
	\end{equation}
	We may assume without loss of generality that $\lambda_\rho :=\inf_{B_{2\rho}} u_\rho>0$ and define $v_\rho(x) = \min\lbrace u_\rho(x)/\lambda_{\rho}, 1 \rbrace$. By testing the equation $\mathcal{L}u_{\rho}=\mu_{\rho}$ with $v_{\rho}$, we obtain
	\begin{align} \label{comp1}
		\mathcal{E}(u_{\rho}, v_{\rho})=\int_{B_{8\rho}} v_{\rho} \,\mathrm{d}\mu_\rho \leq  \mu_{\rho}(\overline{B_{\rho}})
	\end{align}
	since the support of $\mu_{\rho}$ is contained in $\overline{B_{\rho}}$. Moreover, by using
	\begin{equation*}
		g(|D^sv_{\rho}|) |D^sv_{\rho}| \leq
		g(|D^s(u_{\rho}/\lambda_{\rho})|)
		\frac{D^su_{\rho}}{|D^su_{\rho}|} D^sv_{\rho} \leq \lambda_\rho^{-(q-1)} g(|D^su_\rho|) \frac{D^su_{\rho}}{|D^su_{\rho}|} D^sv_{\rho},
	\end{equation*}
	we obtain
	\begin{equation} \label{eq:vv}
		\mathcal{E}(v_{\rho}, v_{\rho}) \leq \lambda_\rho^{-(q-1)} \mathcal{E}(u_{\rho}, v_{\rho}).
	\end{equation}
	Furthermore, we observe that $v_\rho$ is admissible for $(s, G)$-capacity of $\overline{B_\rho}$ with respect to $B_{8\rho}$, and hence
	\begin{equation} \label{comp2}
		\mathrm{cap}_{s,G}(\overline{B_{\rho}}, B_{8\rho}) \leq \varrho_{W^{s, G}(\mathbb{R}^n)}(v_\rho) \leq \Lambda \mathcal{E}(v_{\rho}, v_{\rho}).
	\end{equation}
	By combining \eqref{comp1}, \eqref{eq:vv}, \eqref{comp2}, and \Cref{lem-cap}, we obtain
	\begin{equation}\label{eq-mu-rho}
		\rho^{n}G(\rho^{-s}) \leq C \lambda_{\rho}^{-(q-1)} \mu_{\rho}(\overline{B_{\rho}}).
	\end{equation}
	Since $u_\rho \in V^{s, G}_0(B_{8\rho})$ by \Cref{lem-potential}~(i), its Riesz measure $\mu_\rho$ belongs to the dual space $(V^{s, G}_0(B_{8\rho}))^\ast$. Then a standard argument as in Lemma~5.1 in Kim--Lee--Lee~\cite{KLL23} shows that
\begin{equation}\label{eq-mecap}
		\mu_\rho(E) \leq C \mathrm{cap}_{s,G} (D_\rho \cap E, B_{8\rho})
	\end{equation}
	for every compact set $E \subset B_{8\rho}$. Thus, \eqref{eq-mecap} with $E= \overline{B_\rho}$ and \eqref{eq-mu-rho} implies that
	\begin{equation*}
		\lambda_{\rho}^{q-1} \leq C \frac{\mathrm{cap}_{s, G}(D_{\rho}, B_{2\rho})}{\rho^{n-s}g(\rho^{-s})}.
	\end{equation*}
If there is a constant $\alpha \in (0,1)$ such that
	\begin{equation*}
		\liminf_{\rho \to 0}\frac{\mathrm{cap}_{s, G}(D_{\rho}, B_{2\rho})}{\rho^{n-s}g(\rho^{-s})} \geq 2\alpha,
	\end{equation*}
then we can find a constant $\rho_0>0$ such that $\mathrm{cap}_{s, G}(D_{\rho}, B_{2\rho}) \geq \alpha\rho^{n-s}g(\rho^{-s})$ for all $\rho \in (0, \rho_0)$, or
	\begin{equation*}
		 \rho^s g^{-1}\left(\frac{\mathrm{cap}_{s, G}(D_{\rho}, B_{2\rho}) }{\rho^{n-s}}\right) \frac{1}{\rho} \geq \frac{C\alpha^{1/(p-1)}}{\rho} \quad\text{for all } \rho \in (0, \rho_0),
	\end{equation*}
	which contradicts to \eqref{eq-wieint}. Thus, we have
	\begin{equation}\label{eq-liminf-cap}
	\liminf_{\rho \to 0}\frac{\mathrm{cap}_{s, G}(D_{\rho}, B_{2\rho})}{\rho^{n-s}g(\rho^{-s})}=0,
	\end{equation}
	which in turn implies \eqref{eq-necessary-cliam}.
	
We next estimate the second term on the right-hand side of \eqref{pt}. By using \eqref{eq-mecap}, we have
	\begin{equation*}
		\begin{split}
			{\bf W}^{\mu_{\rho}}_{s, G}(x_0, 4\rho)
			&\leq \int_0^{4\rho} t^s g^{-1} \left(\frac{\mu_\rho(\overline{B_{t}})}{t^{n-s}}\right) \frac{\mathrm{d}t}{t} \\
			&\leq C\int_0^{4\rho} t^s g^{-1} \left(\frac{\mathrm{cap}_{s, G}(D_t, B_{8\rho}) }{t^{n-s}}\right) \frac{\mathrm{d}t}{t}\\
			&\leq C\int_0^{4\rho} t^s g^{-1} \left(\frac{\mathrm{cap}_{s, G}(D_t, B_{2t}) }{t^{n-s}}\right) \frac{\mathrm{d}t}{t}.
		\end{split}
	\end{equation*}
	Thus, it follows from \eqref{eq-wieint} that ${\bf W}_{s, G}^{\mu_{\rho}}(x_0, 4\rho) \to 0$ as $\rho \to 0$.

	It remains to estimate the tail term in \eqref{pt}. Since $u_\rho \in V^{s, G}_0(B_{8\rho})$, we have
	\begin{equation*}
		\mathrm{Tail}_g(u_{\rho}; x_0, 2\rho) =(2\rho)^s g^{-1}\left( (2\rho)^{s} \int_{B_{8\rho} \setminus B_{2\rho}} g\left(\frac{u_\rho(y)}{|y-x_0|^{s}}\right) \frac{\mathrm{d}y}{|y-x_0|^{n+s}} \right).
	\end{equation*}
	It follows from \Cref{lem-potential}~(iii) and the ellipticity condition that
	\begin{equation*}
	\Lambda \varrho_{W^{s, G}(\mathbb{R}^n)}(v) \geq \mathcal{E}(v) \geq \mathcal{E}(u_\rho)
	\end{equation*}
	for all functions $v \in W_0(D_\rho, B_{8\rho})$. By taking the infimum over $v \in W_0(D_\rho, B_{8\rho})$, we obatin
	\begin{equation*}
		\Lambda \, \mathrm{cap}_{s, G}(D_\rho, B_{8\rho}) \geq \frac{1}{\Lambda} \int_{B_{10\rho} \setminus B_{8\rho}} \int_{B_{8\rho} \setminus B_{2\rho}} G(|D^su_{\rho}|) \,\frac{\mathrm{d}y \,\mathrm{d}x}{|x-y|^n}.
	\end{equation*}
	For $x \in B_{10\rho} \setminus B_{8\rho}$ and $y \in B_{8\rho} \setminus B_{2\rho}$, it holds that $|x-y| \leq 18\rho \leq 9|y-x_0|$, and hence
	\begin{equation*}
		\begin{split}
			\mathrm{cap}_{s, G}(D_\rho, B_{8\rho}) 
			&\geq C \int_{B_{10\rho} \setminus B_{8\rho}} \int_{B_{8\rho} \setminus B_{2\rho}} G\left(\frac{u_{\rho}(y)}{|y-x_0|^s}\right) \frac{\mathrm{d}y \,\mathrm{d}x}{|y-x_0|^n} \\
			&\geq C \rho^{n} \int_{B_{8\rho} \setminus B_{2\rho}} G\left(\frac{u_{\rho}(y)}{|y-x_0|^s}\right) \frac{\mathrm{d}y}{|y-x_0|^n}.
		\end{split}
	\end{equation*}
	By applying \eqref{eq-alg}, for any $\varepsilon > 0$, we have
	\begin{equation*}
		\begin{split}
			&\int_{B_{8\rho} \setminus B_{2\rho}} g\left(\frac{u_{\rho}(y)}{|y-x_0|^{s}}\right) \frac{\mathrm{d}y}{|y-x_0|^{n+s}}\\
			&\quad \leq \varepsilon^{-1}  \int_{B_{8\rho} \setminus B_{2\rho}} G\left(\frac{u_{\rho}(y)}{|y-x_0|^s}\right) \frac{\mathrm{d}y}{|y-x_0|^n}+ \int_{B_{8\rho} \setminus B_{2\rho}} g\left(\frac{\varepsilon}{|y-x_0|^s}\right) \frac{\mathrm{d}y}{|y-x_0|^{n+s}}\\
			&\quad \leq C\varepsilon^{-1}\rho^{-n}\mathrm{cap}_{s, G}(D_\rho, B_{8\rho}) +C \rho^{-s}g(\varepsilon\rho^{-s})
		\end{split}
	\end{equation*}
	Therefore, we conclude that
	\begin{align*}
		\mathrm{Tail}_g(u_\rho; x_0, 2\rho) &\leq C\rho^s g^{-1}\left( \varepsilon^{-1} \frac{\mathrm{cap}_{s, p}(D_\rho, B_{2\rho})}{\rho^{n-s}} + g(\varepsilon\rho^{-s}) \right) \\
		&\leq C\rho^s g^{-1}\left( \varepsilon^{-1} \frac{\mathrm{cap}_{s, p}(D_\rho, B_{2\rho})}{\rho^{n-s}}\right)+C\varepsilon.
	\end{align*}
	Taking $\varepsilon$ sufficiently small and then using \eqref{eq-liminf-cap}, we can make $\mathrm{Tail}_g(u_\rho; x_0, 2\rho)$ as small as we want.
	
	Thus, there exists a small $\rho > 0$ such that $u_\rho(x_0) < 1$ as required.
\end{proof}

\noindent
		{\bf Acknowledgments.} The authors would like to express their sincere gratitude to the anonymous referees who provided valuable comments and suggestions on the earlier version, which improved the clarity and understanding of the manuscript.

\begin{appendix}

\section{Algebraic inequalities}\label{sec-ineq}

In this section we provide several algebraic inequalities. The following two lemmas are algebraic inequalities that are used in the proof of the upper potential estimate in \Cref{sec-upper}.

\begin{lemma}\label{lem-ineq1}
Let $v$ be a nonnegative function and $\lambda, \alpha >0$. There exists a constant $C > 0$, depending only on $p$, $q$, and $\alpha$, such that
\begin{align}\label{eq-ineq1}
\begin{split}
|D^sW(v)|^p
&\leq \frac{C}{R^{sp}} \left|\fint_{v(y)}^{v(x)} \left( \frac{\bar{g}(\lambda + t)}{\bar{g}(\lambda)} \right)^{\alpha p} {\bf 1}_{\{v(x) \neq v(y)\}}(x, y) \,\mathrm{d}t\right| \\
&\quad+ \frac{C}{R^{sp}} G(R^s|D^sv|) \left| \fint_{v(y)}^{v(x)} \left( \frac{\bar{g}(\lambda+t)}{\bar{g}(\lambda)} \right)^{\alpha p} \frac{\mathrm{d}t}{G(\lambda +t)}\right|,
\end{split}
\end{align}
where $W(t) = (\bar{g}(\lambda+t)/\bar{g}(\lambda))^{\alpha}-1$.
\end{lemma}

\begin{proof}
We may assume without loss of generality that $v(x)>v(y)$. By using Jensen's inequality and \eqref{eq-bar-g-pq}, we have
\begin{align*}
|D^s W(v)|^p
&= \left( D^sv \fint_{v(y)}^{v(x)} \alpha \frac{\bar{g}^{\alpha-1}(\lambda + t) \bar{g}'(\lambda + t)}{\bar{g}^{\alpha}(\lambda)} \,\mathrm{d}t \right)^{p} \\
&\leq \left( \frac{\alpha (q-1)}{R^{s}\bar{g}^{\alpha}(\lambda)} \right)^{p} \fint_{v(y)}^{v(x)} \frac{\bar{g}^{\alpha p}(\lambda + t)}{(\lambda+t)^{p}} |R^s D^sv|^p \,\mathrm{d}t.
\end{align*}
Let us define $H(t):=G(t^{1/p})$. Then by applying Young's inequality and Lemma~2.2 in \cite{KL23}, we obtain
\begin{equation*}
\frac{\bar{g}(\lambda +t)}{(\lambda+t)^{p-1}} |R^s D^sv|^p \leq H^{\ast}\left( \frac{G(\lambda + t)}{(\lambda + t)^{p}} \right) + H(|R^sD^sv|^p) \leq C(G(\lambda + t) + G(R^s|D^sv|)).
\end{equation*}
Thus, the desired inequality \eqref{eq-ineq1} follows by combining the above inequalities.
\end{proof}

\begin{lemma}\label{lem-ineq2}
Let $\lambda >0$ and $\delta>1$. Let $u$ be a function defined in $B_R$ and let $\eta$ be a function satisfying $0\leq \eta\leq 1$. There exist constants $c, C > 0$, depending only on $p$, $q$, and $\delta$, such that
\begin{align}\label{eq-ineq2}
\begin{split}
&g(|D^su|) \frac{D^su}{|D^su|} D^s(\Phi(v)\eta) \\
&\geq c\frac{1}{R^{s}} G(|D^su_+|) \left( \fint_{v(y)}^{v(x)} \left( \frac{\bar{g}(\lambda)}{\bar{g}(\lambda+t)} \right)^{\frac{\delta-1}{q-1}} \frac{\mathrm{d}t}{\lambda + t} \right) (\eta(x) \land \eta(y))^{q} \\
&\quad - C \frac{\bar{g}(\lambda)}{R^{s}} \left( \frac{\bar{g}(\lambda+\max\{v(x), v(y)\})}{\bar{g}(\lambda)} \right)^\delta ((R^{s}|D^s\eta|)^{p} \lor (R^{s}|D^s\eta|)^q) {\bf 1}_{\lbrace v(x)\neq v(y) \rbrace}(x, y),
\end{split}
\end{align}
where $v=u_+/R^s$ and $\Phi(t) = 1-(\bar{g}(\lambda)/\bar{g}(\lambda+t))^{\frac{\delta-1}{q-1}}$.
\end{lemma}

\begin{proof}
We may assume $u(x)>u(y)$ without loss of generality. Let $A$ denote the left-hand side of \eqref{eq-ineq2}.

Let us first consider the case $\eta(x)/\eta(y) \in [1/2,2]$. By using \eqref{eq-bar-g-pq}, we have
\begin{align*}
D^s(\Phi(v)\eta)
&= \left( \int_{v(y)}^{v(x)} \Phi'(t) \,\mathrm{d}t \right) \frac{\eta^q(x)}{|x-y|^s} + \Phi(v(y)) D^s(\eta^q) \\
&\geq \frac{(\delta-1)(p-1)}{q-1} \left( \int_{v(y)}^{v(x)} \left( \frac{\bar{g}(\lambda)}{\bar{g}(\lambda+t)} \right)^{\frac{\delta-1}{q-1}} \frac{\mathrm{d}t}{\lambda + t} \right) \frac{\eta^q(x)}{|x-y|^s} \\
&\quad - q\Phi(v(y))(\eta(x) \lor \eta(y))^{q-1} |D^s\eta|.
\end{align*}
Since $\eta(y) \leq 2\eta(x)$ and $u(x)-u(y) \geq u_+(x)-u_+(y)$, we obtain
\begin{align*}
A
&\geq \frac{(\delta-1)p(p-1)}{(q-1)R^{s}} G(D^su_+) \left( \fint_{v(y)}^{v(x)} \left( \frac{\bar{g}(\lambda)}{\bar{g}(\lambda+t)} \right)^{\frac{\delta-1}{q-1}} \frac{\mathrm{d}t}{\lambda + t} \right) \eta^q(x) \\
&\quad - q^{2} 2^{q-1} \bar{g}(D^su) |D^s\eta| \Phi(v(y)) \eta^{q-1}(x) =: I_{1} + I_{2},
\end{align*}
where we used \eqref{eq-pq} and \eqref{eq-bar-g-comp}. We claim that
\begin{equation*}
\bar{g}(D^su) \frac{|D^s\eta|}{\eta(x)} \Phi(v(y)) \leq \varepsilon G(D^su_+) + \bar{g}\left( \frac{1}{\varepsilon} \frac{|D^s\eta|}{\eta(x)} \right) \frac{|D^s\eta|}{\eta(x)}
\end{equation*}
for any $\varepsilon > 0$. Indeed, it is trivial when $u(y) \leq 0$ since $\Phi(v(y))=0$ in this case. If $u(y) > 0$, then it follows from $u(x)-u(y) \leq u_+(x)-u_+(y)$, $\Phi(v(y))\leq 1$, and the inequality \eqref{eq-alg} with $g$ replaced by $\bar{g}$. We now take
\begin{equation*}
\varepsilon = \frac{(\delta-1)p(p-1)}{q^{2}(q-1) 2^{q}R^{s}} \fint_{v(y)}^{v(x)} \left( \frac{\bar{g}(\lambda)}{\bar{g}(\lambda+t)} \right)^{\frac{\delta-1}{q-1}} \frac{\mathrm{d}t}{\lambda + t}
\end{equation*}
so that
\begin{equation*}
I_{2} \geq -\frac{1}{2} I_{1} - q^{2} 2^{q-1} \bar{g}\left( \frac{1}{\varepsilon} \frac{|D^s\eta|}{\eta(x)} \right) |D^s\eta| \eta^{q-1}(x).
\end{equation*}
Since $\eta(x) \leq 1$ and
\begin{equation*}
\varepsilon \geq \frac{c}{R^{s}} \left( \frac{\bar{g}(\lambda)}{\bar{g}(\lambda+v(x))} \right)^{\frac{\delta-1}{q-1}} \frac{1}{\lambda + v(x)},
\end{equation*}
we have from \Cref{lem-G} that
\begin{align*}
&\bar{g}\left( \frac{1}{\varepsilon} \frac{|D^s\eta|}{\eta(x)} \right) \frac{|D^s\eta|}{\eta(x)} \eta^{q-1}(x) \\
&\leq \frac{C}{R^{s}} \bar{g}\left( \left( \frac{\bar{g}(\lambda+v(x))}{\bar{g}(\lambda)} \right)^{\frac{\delta-1}{q-1}} (\lambda + v(x)) \right) ((R^s|D^s\eta|)^p \lor (R^s|D^s\eta|)^q) \\
&\leq \frac{C}{R^{s}} \bar{g}(\lambda) \left( \frac{\bar{g}(\lambda+v(x))}{\bar{g}(\lambda)} \right)^{\delta} ((R^s|D^s\eta|)^p \lor (R^s|D^s\eta|)^q).
\end{align*}
This finishes the proof for the case $\eta(x) / \eta(y) \in [1/2, 2]$.

Let us next consider the case $\eta(x) / \eta(y) \not\in [1/2, 2]$. In this case, we have
\begin{equation}\label{eq-eta}
	\eta(x) \lor \eta(y) \leq 2|\eta(x)-\eta(y)|.
\end{equation}
 We begin with
\begin{equation*}
	A=g(D^su) \frac{\Phi(v(x)) \eta^q(x)}{|x-y|^{s}}-g(D^su) \frac{\Phi(v(y)) \eta^q(y)}{|x-y|^{s}} =: J_1+J_2.
\end{equation*}
We claim that $J_1$ and $J_2$ are estimated from below by the terms on the right-hand side of \eqref{eq-ineq2}, respectively. Indeed, for $J_1$, we use \eqref{eq-bar-g-pq} to obtain
\begin{equation*}
	\int_{v(y)}^{v(x)} \left( \frac{\bar{g}(\lambda)}{\bar{g}(\lambda+t)} \right)^{\frac{\delta-1}{q-1}} \frac{\mathrm{d}t}{\lambda + t} \leq C \int_{v(y)}^{v(x)} \Phi'(t) \,\mathrm{d}t \leq C\Phi(v(x)).
\end{equation*}
Then it follows from \eqref{eq-pq} and $u(x)-u(y) \geq u_+(x)-u_+(y)$ that
\begin{align*}
	J_1 &\geq \frac{C}{R^s}g(D^su) D^su_+ \left( \fint_{v(y)}^{v(x)} \left( \frac{\bar{g}(\lambda)}{\bar{g}(\lambda+t)} \right)^{\frac{\delta-1}{q-1}} \frac{\mathrm{d}t}{\lambda + t} \right) \eta^q(x) \\
	&\geq \frac{C}{R^s}G(D^su_+) 
	\left( \fint_{v(y)}^{v(x)} \left( \frac{\bar{g}(\lambda+t)}{\bar{g}(\lambda)} \right)^{-\frac{\delta-1}{q-1}} \frac{\mathrm{d}t}{\lambda + t} \right) (\eta(x) \land \eta(y))^q,
\end{align*}
as desired.

On the other hand, we utilize \eqref{eq-bar-g-comp}, \Cref{lem-G}, and $\Phi \leq 1$ to have
\begin{equation*}
	g(D^su) \Phi(v(y)) \leq C \bar{g}\left(\frac{R^sv(x)}{|x-y|^s}\right) \Phi(v(y)) \leq C \left(\frac{R^s}{|x-y|^s}\right)^{q-1}\bar{g}(v(x)).
\end{equation*}
Since $\delta>1$, we obtain
\begin{equation*}
\frac{\bar{g}(v(x))}{\bar{g}(\lambda)} \leq \frac{\bar{g}(\lambda+v(x))}{\bar{g}(\lambda)} \leq \left( \frac{\bar{g}(\lambda+v(x))}{\bar{g}(\lambda)} \right)^{\delta},
\end{equation*}
which together with \eqref{eq-eta} yields
\begin{align*}
	J_2 &\geq -C \bar{g}(\lambda) \left(\frac{\bar{g}(\lambda+v(x))}{\bar{g}(\lambda)}\right)^{\delta} \left(\frac{R^s}{|x-y|^s}\right)^{q-1} \frac{\eta^q(y)}{|x-y|^s} \\
	&\geq -C \frac{\bar{g}(\lambda)}{R^s} \left(\frac{\bar{g}(\lambda+v(x))}{\bar{g}(\lambda)}\right)^{\delta} (R^s|D^s\eta|)^q
\end{align*}
for some $C=C(p, q)>0$.
\end{proof}

\begin{lemma}\label{lem-ineq3}
	Let $a_1, a_2, b_1, b_2 \in \mathbb{R}$. Then
\begin{equation}\label{eq-ineq3}
	G(|a_1 \lor a_2-b_1 \lor b_2 |)+G(|a_1 \land a_2-b_1 \land b_2|) \leq G(|a_1-b_1|)+G(|a_2-b_2|).
\end{equation}
\end{lemma}

\begin{proof}
We may assume without loss of generality that $a_1 \geq a_2$. If $b_1 \geq b_2$, then the desired inequality \eqref{eq-ineq3} is trivial. Otherwise, \eqref{eq-ineq3} holds if and only if the function
\begin{equation*}
\varphi(t):= G(|t-b_1|)-G(|t-b_2|)
\end{equation*}
is non-decreasing. Indeed, we have
	\begin{align*}
		\varphi'(t)&=g(|t-b_1|)\frac{t-b_1}{|t-b_1|}-g(|t-b_2|)\frac{t-b_2}{|t-b_2|} \\
		&=\begin{cases}
			g(t-b_1)-g(t-b_2) \quad &\text{if $t \geq b_2$}\\
			g(t-b_1)+g(b_2-t) \quad &\text{if $b_1 \leq t < b_2$}\\
			g(b_2-t)-g(b_1-t) \quad &\text{if $t < b_1$}
		\end{cases}
		\\  &\geq0.
	\end{align*}
\end{proof}

\end{appendix}


\end{document}